\documentclass[12pt,reqno]{amsart}
\usepackage{amsmath,amstext,amssymb,amscd}
\usepackage{verbatim}
\usepackage{enumerate}
\usepackage{mathrsfs}

\usepackage{amsthm}

\newtheorem{theorem}{Theorem}[section]
\newtheorem{lemma}[theorem]{Lemma}
\newtheorem{corollary}[theorem]{Corollary}
\newtheorem{proposition}[theorem]{Proposition}

\theoremstyle{definition}

\theoremstyle{remark}
\newtheorem{remark}[theorem]{Remark}

\numberwithin{equation}{section}

\begin{document}

\author[Y. Guo]{Yanqiu Guo}
\address[Y. Guo]
{Department of Computer Science and Applied Mathematics \\  Weizmann Institute of Science\\
Rehovot 76100, Israel.} \email{yanqiu.guo@weizmann.ac.il}

\author[E. S. Titi]{Edriss S. Titi}
\address[E. S. Titi]
{Department of Mathematics, Texas A\&M University, 3368 TAMU,
 College Station, TX 77843-3368, USA.  {\bf ALSO},
  Department of Computer Science and Applied Mathematics, Weizmann Institute
  of Science, Rehovot 76100, Israel.} \email{titi@math.tamu.edu and
  edriss.titi@weizmann.ac.il}

\title[Backward behavior of dissipative evolution equations]
{On the backward behavior of some dissipative evolution equations}
\date{February 25, 2015}

\maketitle

\begin{abstract}
We prove that every solution of a KdV-Burgers-Sivashinsky type equation blows up in the energy space, backward in time, provided the solution does not belong to the global attractor. This is a phenomenon contrast to the backward behavior of the periodic 2D Navier-Stokes equations studied by Constantin-Foias-Kukavica-Majda \cite{backward-NSE}, but analogous to the backward behavior of the Kuramoto-Sivashinsky equation discovered by Kukavica-Malcok \cite{KS-backward-05}. Also we study the backward behavior of solutions to the damped driven nonlinear Schr\"odinger equation, the complex Ginzburg-Landau equation, and the hyperviscous Navier-Stokes equations. In addition, we provide some physical interpretation of various backward behaviors of several perturbations of the KdV equation by studying explicit cnoidal wave solutions. Furthermore, we discuss the connection between the backward behavior and the energy spectra of the solutions. The study of backward behavior of dissipative evolution equations is motivated by the investigation of the Bardos-Tartar conjecture stated in \cite{Bardos-Tartar}.
\end{abstract}

\section{Introduction}
Consider a  KdV-Burgers-Sivashinsky (KBS) type equation on the torus $\mathbb T=[-\frac{L}{2},\frac{L}{2}]$:
\begin{align}\label{visKdV}
u_t-\nu u_{xx}+uu_x-\beta u+\gamma u_{xxx}=f, \;\;x\in \mathbb T,
\end{align}
with the initial condition $u(0)=u_0$.
We assume that $\nu\geq 0$, $\beta\in \mathbb R$, and $\gamma \not = 0$. Also, we suppose that the time independent forcing $f$ and the initial data  $u_0$ both have spatial mean value zero,  i.e., $\int_{\mathbb T} f dx=\int_{\mathbb T} u_0 dx=0$, and hence $\int_{\mathbb T} u(x,t) dx=0$ for all time $t$ within the lifespan of the solution $u$.

We define the spaces
\begin{align*}
&H:=\{\varphi\in L^2_{per}(\mathbb T): \int_{\mathbb T} \varphi dx =0\},  \notag\\
&V:=\{\varphi\in L^2_{per}(\mathbb T): \varphi_x\in L^2_{per}(\mathbb T), \;\int_{\mathbb T} \varphi dx =0\}.
\end{align*}

We call a solution $u(t)$ of (\ref{visKdV}) \emph{global} if $u(t)$ is defined for all $t\in \mathbb R$. The main result of this paper states that, if $\nu>0$, then there does not exist a global solution of (\ref{visKdV}) in the energy space $H$ unless $u$ belongs to the global attractor. Specifically, we have the following theorem.

\begin{theorem} \label{main}
Consider the following two cases:
\begin{enumerate}[(i)]
\item Assume $\nu=0$ and $f\in H$. For any $u_0\in H$, (\ref{visKdV}) has a unique global solution $u(t)\in H$ for all $t\in \mathbb R$.
\item  Assume $\nu>0$ and $f\in V'$. Let $u(t): [0,\infty)\rightarrow H$ be a solution of (\ref{visKdV}) which does not belong to the global attractor. Then $u(t)$ cannot be extended to a global solution for all $t\in \mathbb R$.
\end{enumerate}
\end{theorem}

The study of the behavior of solutions to (\ref{visKdV}) backward in time is motivated by the pioneering work \cite{backward-NSE}, where Constantin, Foias, Kukavica and Majda investigated the backward behavior of the solution to the 2D periodic Navier-Stokes equations (NSE). Indeed, they showed that the set of initial data for which the solution exists for all negative time and has exponential growth is rather rich (dense in $\mathcal H$ with the topology of $(\dot H^1_{per})'$ where $\mathcal H=\{u\in L^2_{per}(\mathbb T^2)^2: \text{div} \,u=0, \int_{\mathbb T^2} u dx=0   \}$), which is a quite remarkable result since it indicates that the backward behavior of the 2D periodic NSE is closer to the corresponding linear dissipative equation. Also their result provided a partial positive answer to a conjecture in \cite{Bardos-Tartar}, where Bardos and Tartar conjectured that, for the 2D periodic NSE, the solution semi-flow $S(t)\mathcal H$ is dense in the phase space $\mathcal H$ for any fixed $t>0$.
On the other hand, the 2D Euler equations are globally well-posed forward and backward in time, and the energy is conserved for smooth solutions. The result in \cite{backward-NSE} may be rephrased as, if we add a viscosity to the 2D periodic Euler equations, which gives the NSE, then there is still a rich set of initial data such that the solution can be extended to all negative times with the energy growing exponentially, i.e., adding a viscosity to the periodic 2D Euler equations fails to drive these solutions to blow up backward in finite time.

We remark that a corollary of Theorem \ref{main} (i.e. Corollary \ref{corr} in section \ref{sec-2}) implies that the solution semi-flow $S(t)H$ of the KBS equation (\ref{visKdV}) with $\nu>0$ is not dense in $H$ for any $t>0$. This result is contrast to the Bardos-Tartar conjecture on the NSE mentioned above \cite{Bardos-Tartar}.

Notice that, if we remove the viscosity from the KBS equation (\ref{visKdV}), i.e. $\nu=0$, then it becomes
\begin{align}\label{KdV}
u_t-uu_x-\beta u+\gamma u_{xxx}=f, \;\;x\in \mathbb T,
\end{align}
 which is a KdV type equation with a linear term $-\beta u$. It is well-known that the real-valued KdV equation is globally well-posed in $H$ for $t\in \mathbb R$ (see, e.g., \cite{B-I-T,Bou-2}, and references therein; see also \cite{Guo-Simon-Titi} for a similar result concerning a coupled system of KdV). If $\beta<0$, then forward in time, (\ref{KdV}) is a weakly damped KdV equation that possesses a global attractor in $H$ \cite{Ghi,Gou}. On the other hand, if $\beta>0$, then $-\beta u$ plays the role of a source term forward in time as well as a weak damping backward in time.  If we consider (\ref{KdV}) backward in time with $\beta>0$, it is a dissipative system possessing a global attractor, which seems more ``stable" than 2D periodic Euler equations. So it could be reasonable to guess that the solution of the KBS equation (\ref{visKdV}) (i.e. (\ref{KdV}) plus a viscosity) might have a similar behavior backward in time as the 2D periodic NSE. However, such intuition turns out to be completely false. Indeed, Theorem \ref{main} tells us that no solutions of (\ref{visKdV}), with $\nu>0$, can be extended to all negative time unless they belong to the global attractor. Here, we would like to remark that the different backward behavior of the KBS equation (\ref{visKdV}) and the 2D periodic NSE might be related to the rate of the energy cascade to small scales for the Burgers and the KdV equations as well as for the Euler equations, respectively.
 Furthermore, we stress that, for the backward blow-up of the KBS equation (\ref{visKdV}), the viscosity plays a vital role, contrast to the backward blow-up of the viscous Burgers equation. Indeed, the inviscid Burgers equation can also blow up backward in finite time, but the KBS equation without viscosity (i.e. the KdV type equation (\ref{KdV})) is globally well-posed for all $t\in \mathbb R$.

The proof of Theorem \ref{main} follows the idea presented in \cite{KS-backward-05}. It is shown in \cite{KS-backward-05} that all solutions of the Kuramoto-Sivashinsky equation off the global attractor blow up backward in finite time. Notice that, the Kuramoto-Sivashinsky equation involves a hyperviscosity tending to drive the energy increasing very fast backward in time. However, the KBS equation (\ref{visKdV}) contains a regular viscosity, and a dispersion distributing the energy, which seems closer to the NSE. Nonetheless, the energy of the KBS equation still increases too fast and blows up backward in finite time, like the Kuramoto-Sivashinsky equation, due to the negative viscosity. It can be viewed as adding the dispersion term $u_{xxx}$ to the viscous Burgers equation fails to prevent solutions blowing up as the time goes backwards.

In addition to the study of the KBS equation (\ref{visKdV}), we investigate the backward behavior of solutions to some other typical dissipative equations. Specifically, we show that any solution off the global attractor, to the damped driven nonlinear Schr\"odinger equation, grows exponentially fast backward in time. Also, we prove that all solutions to the complex Ginzburg-Landau equation of certain parameter regime, blow up backward in finite time, unless they belong to the global attractor. Moreover, we discuss the asymptotic behavior of solutions to a hyperviscous NSE, forward and backward in time.

The viscous KdV equation is a special case of the KBS equation by taking $\beta=\gamma=0$ and $f=0$ in (\ref{visKdV}). It is well-known that the classical periodic KdV equation has explicit cnoidal wave solutions. Also the solutions of the KdV are connected to the spectrum problem of certain Schr\"odinger operator. These nice properties of the KdV equation provide us an opportunity to employ the modulation asymptotic analysis to study the backward behavior of the viscous KdV equation. Indeed, we give an evidence that the amplitude of the cnoidal wave solution of the viscous KdV equation approaches infinity backward in finite time, which is consistent with our result on the backward blow-up of the energy. Furthermore, we attempt to analyze the underlying mechanism for the backward blow-up of the viscous KdV equation as well as the backward non-blowup phenomenon of the 2D periodic NSE, by using the energy spectra of solutions together with the Kolmogorov theory of turbulence. These discussions from the physical point of view are novelties of the manuscript and are major motivation for writing this article.

Now we briefly survey some related works in the literature concerning the backward behaviors of dissipative evolution equations. It is proved in \cite{Vukadinovic-02} that the 2D periodic viscous Camassa-Holm equations (also known as Navier-Stokes-$\alpha$ model \cite{Foias-Holm-Titi}) has an analogous backward behavior as the 2D periodic NSE. Some different phenomenon has been discovered in \cite{Dascaliuc-03} for Burgers' original model for turbulence which consists of a PDE coupled with a nonlocal ODE, and it was shown for this system, there are three possible behaviors of a solution as $t\rightarrow -\infty$: it can be globally bounded, grow exponentially fast, or grow faster than any exponential.

Recall that the 2D periodic NSE reads
\begin{align}  \label{2dNSE}
 \begin{cases}
 u_t-\nu \Delta u+(u\cdot \nabla)u+\nabla p=f, \;\;x\in \mathbb T^2 \\
 \text{div\,} u=0,
 \end{cases}
 \end{align}
 with $u(0)=u_0$. A main ingredient in studying the backward behavior of solutions to (\ref{2dNSE}) is a pair of orthogonality relations of the form $((u\cdot \nabla)v,v)_{L^2}=0$ and $((u\cdot \nabla)u, \Delta u)_{L^2}=0$, provided $u$, $v$ are periodic in two dimensional space and divergence free. Such pair of orthogonality relations lead to a pair of energy and enstrophy formula: $\frac{1}{2}\frac{d}{dt}|u|_{L^2}^2+\nu|\nabla u|_{L^2}^2=(f,u)_{L^2}$ and $\frac{1}{2}\frac{d}{dt}|\nabla u|_{L^2}^2+\nu|\Delta u|_{L^2}^2=-(f,\Delta u)_{L^2}$, which are identical to the corresponding estimates for the linear Stokes equations (i.e. the equation (\ref{2dNSE}) with the nonlinear term vanished), and are the essence for proving that there is a rich set of initial data such that the solution to (\ref{2dNSE}) exists globally for all  $t\in \mathbb R$ (see \cite{backward-NSE}). It is demonstrated in \cite{Lorenz} how to generate a second orthogonality property for Lorenz equations, by finding an alternate linear operator.

In addition, see \cite{Dascaliuc-05,Vukadinovic-04} for some related results on this topic.

Throughout, $|\cdot|$ and $(\cdot,\cdot)$ represent the norm and the inner product in $L^2_{per}(\mathbb T)$ respectively, and $\|\cdot\|$ stands for the norm in $H^1_{per}(\mathbb T)$, i.e., $\|u\|^2=|u_x|^2+|u|^2$ for $u\in H^1_{per}(\mathbb T)$. As usual, we set $A=-\partial_{xx}$. Moreover, we will denote by $c$ a dimensionless constant that may change from line to line.

The paper is organized as follows. In section \ref{main} we shall prove the main rigorous mathematical result of the paper: Theorem \ref{main}. In section \ref{sec-dis}, we investigate the backward behavior of solutions to damped driven nonlinear Schr\"odinger equation, complex Ginzburg-Landau equation, and hyperviscous Navier-Stokes equations, as well as mentioning some open problems. Section \ref{spectra} provides an elegant physical interpretation of the backward behavior of various perturbations of the KdV equation, and also includes some interesting discussion about the relation of the backward behavior of dissipative equations with the Kolmogorov theory of energy spectra of turbulence. The appendix is devoted to proving two results presented in section \ref{sec-dis}.

\bigskip

\section {Proof of Theorem \ref{main}}     \label{sec-2}
To begin with, let us consider the first part of Theorem \ref{main}, i.e., $\nu=0$. Without the viscous term, (\ref{visKdV}) becomes a KdV type equation (\ref{KdV}). The global well-posedness of the real-valued KdV equation $u_t+uu_x+u_{xxx}=0$ in $H$ has been established in Bourgain's seminal paper \cite{Bou-2} (see \cite{B-I-T} for an alternative proof). Note that the energy of the KdV equation is conserved. Adding a linear perturbation $-\beta u$ and a time-independent forcing $f\in H$ to the KdV equation (ending up with equation (\ref{KdV})) destroys the energy conservation, but the global well-posedness in $H$ for all $t\in \mathbb R$ is still valid for (\ref{visKdV}). Indeed, taking the scalar product of (\ref{KdV}) with $u$ yields
\begin{align*}
\frac{1}{2}\frac{d}{dt}|u|^2 -\beta |u|^2 = (f,u).
\end{align*}
By assuming $\beta>0$ and using H\"older's, Young's and Gronwall inequalities, we deduce
\begin{align} \label{fo}
|u(t)|^2 \leq e^{3\beta t} |u_0|^2 + \frac{|f|^2}{3\beta^2} (e^{3\beta t}-1), \text{\;\;for\;\;} t\geq 0,
\end{align}
and
\begin{align}  \label{ba}
|u(t)|^2 \leq e^{\beta t}|u_0|^2 - \frac{1}{\beta^2} (e^{\beta t} -1 ) |f|^2, \text{\;\;for\;\;} t\leq 0.
\end{align}
Notice that, (\ref{fo}) implies that the energy grows at most exponentially fast as $t\rightarrow +\infty$, while (\ref{ba}) shows that as the time $t$ goes backwards, the energy is uniformly bounded. In fact, $\limsup_{t \rightarrow -\infty} u(t) \leq \frac{1}{\beta^2} |f|^2$, and thus all trajectories of equation (\ref{KdV}) approach a global attractor in $H$ as $t\rightarrow -\infty$, provided $\beta>0$. By switching the direction of the time, the same analysis holds true for the case $\beta<0$.

Now we prove the second part of Theorem \ref{main}, i.e., the case $\nu>0$.

\subsection{Asymptotic estimates for the long-time dynamics}
In the second part of Theorem \ref{main}, we assume $f\in V'$. In the proof, for the sake of clarification, we assume $\beta>0$, which is the more interesting case. The case $\beta\leq 0$ can be treated analogously with less effort.

In order to study the behavior of solutions for evolutionary equations forward or backward in time, a natural method is to perform an \emph{a priori} estimate of the energy, and for equation (\ref{visKdV}) the energy is the $L^2$-norm of the solution.
A straightforward way of estimating the energy $|u|$ is to take the scalar product of the KBS equation (\ref{visKdV}) with $u$:
\begin{align*}
\frac{1}{2}\frac{d}{dt}|u|^2+\nu |u_x|^2 -\beta |u|^2 =  (f,u) \leq \frac{\nu}{2} |u_x|^2 + \frac{1}{2\nu} |A^{-\frac{1}{2}}f|^2.
\end{align*}
It follows that
\begin{align} \label{ene}
\frac{d}{dt}|u|^2+(\nu \lambda_1 -2\beta) |u|^2 \leq  \frac{1}{\nu} |A^{-\frac{1}{2}}f|^2,
\end{align}
where the Poincar\'e inequality $|u_x|^2\geq \lambda_1 |u|^2$ has been used, for any $u\in V$. Here, $\lambda_1$ is the first eigenvalue of the operator $A=-\partial_{xx}$, i.e., $\lambda_1=(2\pi/L)^2$.
Unless the viscosity $\nu$ is large (i.e., $\nu>\frac{2\beta}{\lambda_1}$ which is not an interesting case), the energy estimate (\ref{ene}) does not provide an uniform bound on the $L^2$-norm of $u(t)$ for all time $t$. However, intuitively the energy of the KBS equation (\ref{visKdV}) must be uniform bounded since it involves the viscous term $\nu u_{xx}$ acting as a strong dissipation. Notice that, a drawback of the above energy estimate is that it does not take advantage of the nonlinear convection term $uu_x$, which plays the role of a medium transferring energy from lower to higher wave numbers to prevent the growth of low wave number modes. Indeed, in the KBS equation (\ref{visKdV}), $-\beta u$ amplify the lower wave numbers while the diffusion $\nu u_{xx}$ damps the higher wave numbers, and the interaction between these two linear terms are through the energy conducting $uu_x$.

In order to take advantage of the convection $uu_x$, we employ the method of Lyapunov functions (also called the background flow method \cite{Constantin-Doering-95,Cons-Doer-95,Doering-Constantin-92,Doering-Constantin-94}). More precisely, we subtract $u$ by an appropriate gauge function $\phi$ and study $|u-\phi|$, instead of directly estimating the energy $|u|$. This technique was used in \cite{Bronski,Collet-Eckmann,Goodman-94,KS-backward-92, KS-backward-05,KS-84} for the purpose of studying the Kuramoto-Sivashinsky equation (KSE). It is worth mentioning that, in the case of odd functions, which is an invariant space for the KSE (also for equation (\ref{visKdV}) with $\gamma=0$), the gauge function $\phi$ is simpler and was introduced in \cite{KS-84}. On the other hand, the best asymptotic energy estimate with respect to the length $L$ for the KSE was obtained in \cite{Otto-09} (see also \cite{Otto-05}) by using different techniques.

We adopt the gauge function $\phi$ which was used in \cite{Goodman-94, KS-backward-05}. Similar idea can be found in \cite{Bronski,Collet-Eckmann}.
To construct the function $\phi$, we need to do some preparation. In fact, for every $\epsilon\in (0,L/2)$, there exists a periodic non-negative smooth function $b_{\epsilon}: \mathbb R \rightarrow \mathbb [0,\infty)$ with period $L$ and supported in $(-\epsilon,\epsilon)$ such that
\begin{enumerate}
\item $\int_{\mathbb T} b_{\epsilon}(x) dx = L$;
\item $\sup_{x\in \mathbb R} b_{\epsilon}(x) \leq c L/\epsilon$;
\item $|b_{\epsilon}| \leq c L/{\epsilon}^{\frac{1}{2}}$;
\item $|b'_{\epsilon}|\leq c L/\epsilon^{\frac{3}{2}}$.
\end{enumerate}
An example of such function was given in \cite{KS-backward-05}: if define $\eta(x)=L(\epsilon-|x|)/\epsilon^2$ for $|x|\leq \epsilon$, then $b_{\epsilon}$ can be obtained by mollifying and periodically extending $\eta$ to the whole real line. The following Poincar\'e type inequality was proved in \cite{Goodman-94}.

\begin{proposition}  \cite{Goodman-94}          \label{prop1}
 Let $b_{\epsilon}: \mathbb R \rightarrow \mathbb [0,\infty)$ be a periodic function with period $L$
 such that $\emph{supp} \, b_{\epsilon} \cap [-\frac{L}{2},\frac{L}{2}] \subset (-\epsilon,\epsilon)$,
 $\int_{\mathbb T} b_{\epsilon}(x) dx = L$, and
 $\sup_{x\in \mathbb R} b_{\epsilon}(x) \leq c L/\epsilon$. If $\int_{\mathbb T} b_{\epsilon} (x) u(x) dx=0$,
 for some $u\in V$, then
 \begin{align*}
 \int_{\mathbb T} b_{\epsilon}(x) u^2 (x) dx \leq c_0 \epsilon  L \int_{\mathbb T} u_x^2(x) dx, \text{\;\;for some\;\;} c_0>0.
 \end{align*}
 \end{proposition}

As in \cite{Goodman-94,KS-backward-05}, we define a periodic smooth function $\phi_{\alpha,\epsilon}: \mathbb R\rightarrow \mathbb R$ by
\begin{align}  \label{def-phi}
\phi_{\alpha,\epsilon}(x)=\alpha x- \alpha \int_0^x b_{\epsilon}(y)dy,
\end{align}
where $\alpha>0$.
Since $b_{\epsilon}$ is periodic with period $L$ and $\int_{\mathbb T} b_{\epsilon}(x) dx = L$, we see that $\phi_{\alpha,\epsilon}$ is also periodic with period $L$. Furthermore, by using the properties of the function $b_{\epsilon}$, simple calculation gives
\begin{align}
&|\phi_{\alpha,\epsilon}|\leq c\alpha L^{\frac{3}{2}},  \label{phi1} \\
&|\phi'_{\alpha,\epsilon}| \leq c \alpha (L^{\frac{1}{2}}+L/\epsilon^{\frac{1}{2}} ),  \label{phi2} \\
&|\phi''_{\alpha,\epsilon}| \leq c \alpha  L/\epsilon^{\frac{3}{2}}.  \label{phi3}
\end{align}
Also the following simple fact is useful: since $\phi_{\alpha,\epsilon}$ is periodic, its $L^2$-norm has translation invariance
$|\phi_{\alpha,\epsilon}(x)|=|\phi_{\alpha,\epsilon}(x+\xi)|$ for any $\xi\in \mathbb R$.

\begin{theorem} \label{thm-ball}
For sufficiently large time $t\geq t_1$ which depends on $|u_0|$,  the solution $S(t)u_0$ of (\ref{visKdV}) enters a ball in $H$ with the radius $\rho$ satisfying the following asymptotic relation with the length $L$ and the parameters $\nu$, $\beta$, $\gamma$:
\begin{align} \label{asym}
\rho \sim L^{5/2}, \; \rho \sim \frac{1}{\nu^2},\; \rho \sim \beta^{5/2}, \; \rho \sim \gamma.
\end{align}
\end{theorem}

\begin{remark}
Theorem \ref{thm-ball} shows that there exists an absorbing ball for (\ref{visKdV}), so by the classical theory of attractors (see, e.g. \cite{Tbook}), the KBS equation (\ref{visKdV}) possesses a global attractor.
\end{remark}

\begin{proof}
The approach follows the idea in \cite{Goodman-94,KS-backward-05}. We aim to estimate
\begin{align*}
F_{\alpha,\epsilon}(t):= \inf_{\xi \in \mathbb T} \int_{\mathbb T}(u(x,t) - \phi_{\alpha,\epsilon}(x+\xi))^2 dx.
\end{align*}
To this end, for $\xi\in \mathbb R$, we calculate
\begin{align} \label{ba1}
\frac{1}{2}\frac{d}{dt} \int_{\mathbb T} (u(x,t)-\phi_{\alpha,\epsilon}(x+\xi))^2 dx
=\frac{1}{2}\frac{d}{dt} |u|^2-\int_{\mathbb T} u_t(x,t) \phi_{\alpha,\epsilon}(x+\xi) dx.
\end{align}
The first term on the right-hand side of (\ref{ba1}) can be estimated by usual way. In fact, taking the scalar product of (\ref{visKdV}) with $u$ yields
\begin{align} \label{ba2}
\frac{1}{2}\frac{d}{dt} |u|^2&= - \nu |u_x|^2 +\beta |u|^2 + (f,u) \notag\\
&\leq -\frac{3}{4}\nu |u_x|^2 + \beta |u|^2  +   \frac{1}{\nu} |A^{-\frac{1}{2}}f|^2,
\end{align}
by using Cauchy-Schwarz and Young's inequalities.

To calculate the last term on the right-hand side of (\ref{ba1}), we take the scalar product of (\ref{visKdV}) with $\phi_{\alpha,\epsilon}(x+\xi)$,
and notice that $\phi_{\alpha,\epsilon}'(x)=\alpha-\alpha b_{\epsilon}(x)$,
\begin{align}    \label{ba2'}
-(u_t,\phi_{\alpha,\epsilon}(x+\xi))
&=\nu (u_x,\phi'_{\alpha,\epsilon}(x+\xi))
-\frac{1}{2} (u^2,\phi'_{\alpha,\epsilon}(x+\xi))
-\beta (u,\phi_{\alpha,\epsilon}(x+\xi))  \notag\\
&\;\;\;+\gamma (u_x,\phi''_{\alpha,\epsilon}(x+\xi)  )
-(f,\phi_{\alpha,\epsilon}(x+\xi)) \notag\\
&\leq \frac{\nu}{4} |u_x|^2 + \nu |\phi'_{\alpha,\epsilon}|^2-\frac{\alpha}{2} |u|^2+
\frac{\alpha}{2} \int_{\mathbb T}  b_{\epsilon}(x+\xi) u^2(x,t) dx \notag\\
&\;\;\; + \beta (|u|^2 + \frac{1}{4} |\phi_{\alpha,\epsilon}|^2)
+\frac{\nu}{4}|u_x|^2  +   \frac{\gamma^2}{\nu} |\phi''_{\alpha,\epsilon}|^2
+\frac{1}{\nu}|A^{-\frac{1}{2}}f|^2+\frac{\nu}{4} |\phi'_{\alpha,\epsilon}|^2 \notag\\
&\leq \frac{\nu}{2} |u_x|^2-(\frac{\alpha}{2}-\beta) |u|^2 + \frac{1}{\nu}|A^{-\frac{1}{2}}f|^2
+c\nu \alpha^2 (L+L^2/\epsilon)  \notag\\
&\;\;\;+ c \beta \alpha^2 L^3  + \frac{c\gamma^2}{\nu} \alpha^2 L^2 /\epsilon^3
+ \frac{\alpha}{2} \int_{\mathbb T}  b_{\epsilon}(x+\xi) u^2(x,t) dx,
\end{align}
where we have used the estimate (\ref{phi1})-(\ref{phi3}) in the last inequality.

Now, substituting (\ref{ba2}) and (\ref{ba2'}) into (\ref{ba1}) implies
\begin{align} \label{ba0}
&\frac{1}{2}\frac{d}{dt} \int_{\mathbb T} (u(x,t)-\phi_{\alpha,\epsilon}(x+\xi))^2 dx   \notag\\
&\leq -\frac{\nu}{4} |u_x|^2 - (\frac{\alpha}{2}-2\beta)|u|^2
+\frac{2}{\nu}|A^{-\frac{1}{2}}f|^2 +  c\nu \alpha^2 (L+L^2/\epsilon)  \notag\\
&\;\;\;+ c \beta \alpha^2 L^3  + \frac{c\gamma^2}{\nu \epsilon^3} \alpha^2 L^2
 + \frac{\alpha}{2} \int_{\mathbb T}  b_{\epsilon}(x+\xi) u^2(x,t) dx.
\end{align}

Next we intend to apply Proposition \ref{prop1} to estimate the last term in (\ref{ba0}).
Recall that we have defined $F_{\alpha,\epsilon}(t)= \inf_{\xi \in \mathbb T} \int_{\mathbb T} (u(x,t) - \phi_{\alpha,\epsilon}(x+\xi))^2 dx.$
Notice, for each fix $t$, the mapping $\xi \mapsto \int_{\mathbb T} (u(x,t) - \phi_{\alpha,\epsilon}(x+\xi))^2 dx$ is continuous and periodic on $\mathbb R$. Thus, we can assume the minimum of $\int_{\mathbb T} (u(x,t) - \phi_{\alpha,\epsilon}(x+\xi))^2 dx$ occurs at a point $\xi^{\ast}(t)\in \mathbb T$. Then $\xi^{\ast}(t)$ must satisfy
\begin{align*}
\int_{\mathbb T} (u(x,t)-\phi_{\alpha, \epsilon}(x+\xi^{\ast}(t) ) )  \phi_{\alpha,\epsilon}'(x+\xi^{\ast}(t)) dx =0.
\end{align*}
Recall that $\phi_{\alpha,\epsilon}$ is a periodic function, thus $\int_{\mathbb T} \phi_{\alpha, \epsilon}(x+\xi^{\ast}(t) )   \phi_{\alpha,\epsilon}'(x+\xi^{\ast}(t)) dx =0$. It follows that
\begin{align*}
\int_{\mathbb T} u(x,t) \phi_{\alpha,\epsilon}'(x+\xi^{\ast}(t)) dx =0,
\end{align*}
and by (\ref{def-phi}), this is equivalent to
\begin{align*}
\int_{\mathbb T} u(x,t) (1-b_{\epsilon}(x+\xi^{\ast}(t))) dx =0.
\end{align*}
Since we assume $u$ has spatial mean value zero, one has
\begin{align*}
\int_{\mathbb T} u(x,t)b_{\epsilon}(x+\xi^{\ast}(t)) dx =0.
\end{align*}
Consequently, by Proposition \ref{prop1} and keeping in mind that $u$ and $b_{\epsilon}$ are both periodic functions with the period $L$,
we conclude that
\begin{align}   \label{ba0'}
\int_{\mathbb T} u^2(x,t)   b_{\epsilon}   (x+ \xi^{\ast}(t)) dx \leq c_0 \epsilon L |u_x|^2.
\end{align}

Now we compute
\begin{align*}
\frac{d}{dt}_{+} F_{\alpha,\epsilon} (t)
&=\limsup_{s\rightarrow t^+} \frac{F_{\alpha,\epsilon} (s)-F_{\alpha,\epsilon} (t)}{s-t}    \notag\\
&=\limsup_{s\rightarrow t^+} \frac{\int_{\mathbb T}(u(x,s) - \phi_{\alpha,\epsilon}(x+\xi^{\ast}(s)))^2 dx-\int_{\mathbb T}(u(x,t) -\phi_{\alpha,\epsilon}(x+\xi^{\ast}(t)))^2 dx}{s-t}  \notag\\
&\leq \limsup_{s\rightarrow t^+} \frac{\int_{\mathbb T}(u(x,s) - \phi_{\alpha,\epsilon}(x+\xi^{\ast}(t)))^2 dx-\int_{\mathbb T}(u(x,t) -\phi_{\alpha,\epsilon}(x+\xi^{\ast}(t)))^2 dx}{s-t}  \notag\\
&=\left(\frac{d}{dt}   \int_{\mathbb T} (u(x,t) - \phi_{\alpha,\epsilon}(x+\xi)  )^2 dx  \right)_{\xi=\xi^{\ast}(t)} \notag\\
&\leq (-\frac{\nu}{2}+\alpha c_0\epsilon L) |u_x|^2 - (\alpha-4\beta)|u|^2
+\frac{4}{\nu}|A^{-\frac{1}{2}}f|^2 +  c\nu \alpha^2 (L+L^2/\epsilon)  \notag\\
&\;\;\;+ c \beta \alpha^2 L^3  + \frac{c\gamma^2}{\nu \epsilon^3} \alpha^2 L^2,
\end{align*}
where (\ref{ba0}) and (\ref{ba0'}) were used to obtain the last inequality.

By choosing $\alpha\geq 8\beta$, and setting $\epsilon = \frac{\nu}{2\alpha c_0 L}$, we obtain
\begin{align}  \label{ba4}
\frac{d}{dt}_{+} F_{\alpha,\epsilon} (t) \leq -\frac{\alpha}{2} |u|^2 + \frac{4}{\nu} |A^{-\frac{1}{2}}f|^2 +C_0(\alpha)
\end{align}
where $C_0(\alpha)=c\nu\alpha^2 L+ c\alpha^3 L^3+ c\beta \alpha^2 L^3+ \frac{c\gamma^2}{\nu^4} \alpha^5 L^5$.

We shall look for an estimate on $F_{\alpha,\epsilon} (t)$ by using (\ref{ba4}). Indeed,
\begin{align*}
F_{\alpha,\epsilon}(t)&= \inf_{\xi \in \mathbb T} \int_{\mathbb T}(u(x,t) - \phi_{\alpha,\epsilon}(x+\xi))^2 dx \leq 2 |u(t)|^2 + 2 |\phi_{\alpha,\epsilon}|^2.
\end{align*}
It follows that
$$-\frac{\alpha}{2} |u(t)|^2
\leq -\frac{\alpha}{4}F_{\alpha,\epsilon}(t)+\frac{\alpha}{2}  |\phi_{\alpha,\epsilon}|^2
\leq -\frac{\alpha}{4}F_{\alpha,\epsilon}(t)+c\alpha^3 L^3,$$
where (\ref{phi1}) was used.
Substituting the above inequality into (\ref{ba4}) yields
\begin{align*}
\frac{d}{dt}_{+} F_{\alpha,\epsilon} (t) \leq -\frac{\alpha}{4}F_{\alpha,\epsilon}(t)+c \alpha^3 L^3
+\frac{4}{\nu} |A^{-\frac{1}{2}}f|^2 +C_0(\alpha).
\end{align*}
By Gronwall's inequality, we deduce, for $t\geq t_0$,
\begin{align}   \label{tt0}
F_{\alpha,\epsilon}(t) \leq e^{-\frac{\alpha}{4}(t-t_0)} F_{\alpha,\epsilon}(t_0)+\frac{4}{\alpha} (1-e^{-\frac{\alpha}{4}(t-t_0)   })
(c \alpha^3 L^3+\frac{4}{\nu} |A^{-\frac{1}{2}}f|^2 +C_0(\alpha)).
\end{align}

This implies
\begin{align*}
\limsup_{t \rightarrow \infty} F_{\alpha,\epsilon}(t) \leq
\frac{4}{\alpha}
(c \alpha^3 L^3+\frac{4}{\nu} |A^{-\frac{1}{2}}f|^2 +C_0(\alpha)).
\end{align*}
Consequently,
\begin{align*}
\limsup_{t\rightarrow \infty} |u(t)|^2 &\leq \frac{8}{\alpha}
(c \alpha^3 L^3+\frac{4}{\nu} |A^{-\frac{1}{2}}f|^2 +C_0(\alpha))+ 2 |\phi_{\alpha,\epsilon}|^2 \notag\\
&\leq \frac{8}{\alpha}
(c \alpha^3 L^3+\frac{4}{\nu} |A^{-\frac{1}{2}}f|^2 +C_0(\alpha))
\end{align*}
by virtue of (\ref{phi1}), where $\alpha\geq 8\beta$.

Finally, if we choose $\alpha=8\beta$, then for $t$ sufficiently large, $u(t)$ enters a ball in $H$ with the radius
\begin{align} \label{rho}
\rho^2:=c (\beta^2 L^3+\frac{1}{\nu \beta} |A^{-\frac{1}{2}}f|^2+ \nu\beta^2 L+ \beta^3 L^3+ \frac{\gamma^2}{\nu^4} \beta^5 L^5 ).
\end{align}
For simplifying the expression of $\rho$, we can assume $\nu\leq 1$ and $L, \beta, \gamma\geq 1$, then
$$\rho^2=c (\frac{1}{\nu \beta} |A^{-\frac{1}{2}}f|^2+\frac{\gamma^2}{\nu^4} \beta^5 L^5 ).$$
This gives $\rho \sim L^{5/2}, \; \rho \sim \frac{1}{\nu^2},\; \rho \sim \beta^{5/2}, \; \rho \sim \gamma.$ The proof of Theorem \ref{thm-ball} is complete.
\end{proof}

\begin{remark} \label{rem2}
By virtue of (\ref{rho}), we see that the dispersion $\gamma u_{xxx}$ in the KBS equation (\ref{visKdV}) tends to amplify the size of the absorbing ball. In particular, if we remove $\gamma u_{xxx}$ from the KBS, then $\rho \sim L^{3/2}$ (instead of $\rho \sim L^{5/2}$ for the KBS). This agrees with the estimate in Theorem 2 \cite{Goodman-94} for the Burgers-Sivashinsky equation $u_t+uu_x=u+u_{xx}$.
\end{remark}

\smallskip

\begin{remark} \label{rem3}
By switching $t$ and $t_0$ in (\ref{tt0}), we obtain, for $t\leq t_0$,
\begin{align} \label{tt1}
F_{\alpha,\epsilon}(t) \geq e^{\frac{\alpha}{4}(t_0-t)} F_{\alpha,\epsilon}(t_0)-\frac{4}{\alpha} (e^{\frac{\alpha}{4}(t_0-t)}-1)
(c \alpha^3 L^3+\frac{4}{\nu} |A^{-\frac{1}{2}}f|^2 +C_0(\alpha)).
\end{align}
Assume $u(t): \mathbb R \rightarrow H$ is a global solution outside the global attractor. By the theory of attractors (see, e.g. \cite{Tbook}), we know that $|u(t)|$ is not universal bounded for all $t\in \mathbb R$. Therefore, $F_{\alpha,\epsilon}(t)$ is not uniformly bounded on $\mathbb R$. So, we can choose $t_0$ such that $F_{\alpha,\epsilon}(t_0)>\frac{4}{\alpha}
(c \alpha^3 L^3+\frac{4}{\nu} |A^{-\frac{1}{2}}f|^2 +C_0(\alpha))$, and for such $t_0$, (\ref{tt1}) implies that
\begin{align*}
\liminf_{t\rightarrow -\infty} \frac{\log F_{\alpha,\epsilon}(t)}{|t|} \geq \frac{\alpha}{4}.
\end{align*}
It follows that
\begin{align*}
\liminf_{t\rightarrow -\infty} \frac{\log |u(t)|}{|t|} \geq \frac{\alpha}{4}.
\end{align*}
Since $\alpha\geq 8\beta$ is an arbitrary large number, we conclude that the growth rate of $|u(t)|$ as $t\rightarrow -\infty$ is faster than exponential, if we assume $u(t)$ is a global solution on $\mathbb R$ outside the global attractor.
In the next section, we shall show that in fact, for the KBS equation (\ref{visKdV}) with $\nu>0$, there does not exist a global solution outside the global attractor.
\end{remark}

\smallskip

\subsection{Blow-up of solutions backward in finite time}
This subsection is devoted to completing the proof of Theorem \ref{main}. We shall show that, if $\nu>0$, then any solution $u(t): [0,\infty)\rightarrow H$ of (\ref{visKdV}) outside the global attractor cannot be extended to a global solution on $\mathbb R$. Our argument follows the approach in \cite{KS-backward-05}. First let us prove the following lemma.

\begin{lemma} \label{lem-1}
Assume $\nu>0$ and $f\in V'$. Let $\alpha\geq 8\beta$ and $R\geq \left(  \frac{4(\frac{4}{\nu} |A^{-\frac{1}{2}}f|^2 +C_0(\alpha))}{\alpha}  \right)^{\frac{1}{2}}$. Assume that $u(t):[t_1,t_2]\rightarrow H$ is a solution of (\ref{visKdV}) such that $|u(t_1)|\leq 2R$ and $|u(t)|\geq R$ for $t_1 \leq t \leq t_2$. Then
\begin{align*}
t_2-t_1\leq \frac{30}{\alpha}+\frac{c\alpha L^3}{R^2}.
\end{align*}
\end{lemma}

\begin{proof}
Recall  $F_{\alpha,\epsilon}(t)= \inf_{\xi \in \mathbb T} \int_{\mathbb T} (u(x,t) - \phi_{\alpha,\epsilon}(x+\xi))^2 dx.$ It follows that
\begin{align}  \label{ba5}
F_{\alpha,\epsilon}(t_1)\leq 2(|u(t_1)|^2 + |\phi_{\alpha,\epsilon}|^2)  \leq 8R^2+c\alpha^2 L^3,
\end{align}
where we have used (\ref{phi1}).

Notice that the inequality (\ref{ba4}) states
$\frac{d}{dt}_+ F_{\alpha,\epsilon}(t)
\leq -\frac{\alpha}{2} |u|^2 +\frac{4}{\nu} |A^{-\frac{1}{2}}f|^2 +C_0(\alpha) $. Integrating between $t_1$ and $t_2$ and using the fact $|u(t)|\geq R$ for $t_1 \leq t \leq t_2$, we infer
\begin{align}  \label{ba6}
F_{\alpha,\epsilon}(t_2)-F_{\alpha,\epsilon}(t_1)
\leq \left(-\frac{\alpha}{2} R^2
+ (\frac{4}{\nu} |A^{-\frac{1}{2}}f|^2 +C_0(\alpha))\right)(t_2-t_1).
\end{align}
Combining (\ref{ba5}) and (\ref{ba6}) yields
\begin{align}  \label{ba7}
F_{\alpha,\epsilon}(t_2) \leq 8R^2+c\alpha^2 L^3+\left(-\frac{\alpha}{2} R^2
+ (\frac{4}{\nu} |A^{-\frac{1}{2}}f|^2 +C_0(\alpha))\right)(t_2-t_1).
\end{align}

On the other hand, since $F_{\alpha,\epsilon}(t)= \int_{\mathbb T}(u(x,t) - \phi_{\alpha,\epsilon}(x+\xi^{\ast}(t)))^2 dx$, then by (\ref{phi1}) we have
\begin{align} \label{ba8}
F_{\alpha,\epsilon}(t_2)\geq \frac{1}{2} |u(t_2)|^2-|\phi_{\alpha,\epsilon}|^2
\geq \frac{1}{2}R^2-c\alpha^2 L^3.
\end{align}

It follows from (\ref{ba7}) and (\ref{ba8}) that
\begin{align*}
\frac{1}{2}R^2-c\alpha^2 L^3 \leq 8R^2+c\alpha^2 L^3+\left(-\frac{\alpha}{2} R^2
+ (\frac{4}{\nu} |A^{-\frac{1}{2}}f|^2 +C_0(\alpha))\right)(t_2-t_1).
\end{align*}
This implies
\begin{align*}
t_2-t_1 &\leq \frac{\frac{15}{2}R^2+c\alpha^2 L^3}{ \frac{\alpha}{2} R^2
- (\frac{4}{\nu} |A^{-\frac{1}{2}}f|^2 +C_0(\alpha)) } \notag\\
&\leq \frac{30 R^2+c\alpha^2 L^3}{\alpha R^2} \notag\\
&=\frac{30}{\alpha}+\frac{c\alpha L^3}{R^2},
\end{align*}
 due to the fact $\frac{1}{4}\alpha R^2 \geq \frac{4}{\nu} |A^{-\frac{1}{2}}f|^2 +C_0(\alpha)$ from assumption.
\end{proof}

Now we can complete the proof of the second part of Theorem \ref{visKdV}: the blow-up of solutions of the KBS equation (\ref{visKdV}) backward in finite time, if $\nu>0$.

\begin{proof} (Theorem \ref{main} (ii))
Assume that there exists a global solution $u(t):\mathbb R \rightarrow H$ which does not belong to the global attractor. By Remark \ref{rem3}, $|u(t)|$ grows faster than exponentially backward in time. Then, there exists a sequence of times $\{t_j\}_{j=0}^{\infty}$ with $t_{j+1}\leq t_j$, such that
$|u(t_{j+1})|\leq 2^{j+1}R_0$ and  $|u(t)|\geq 2^{j} R_0$ for $t_{j+1} \leq t \leq t_j$, $j=0, 1,2,\cdots$, where $R_0>0$ will be chosen later.

By taking $R=2^j R_0$ in Lemma \ref{lem-1}, we have, if $4^{j-1} R_0^2 \alpha_j  \geq \frac{4}{\nu}|A^{-\frac{1}{2}} f |^2  +C_0(\alpha_j)$ and $\alpha_j\geq 8\beta$, then
\begin{align} \label{star}
\sum_{j=0}^{\infty} (t_{j}-t_{j+1}) \leq 30\sum_{j=0}^{\infty}  \frac{1}{\alpha_j }
+\sum_{j=0}^{\infty}\frac{c\alpha_j L^3}{4^{j} R_0^2}.
\end{align}

Now we carefully select the values of $\alpha_j$ and $R_0$ to force the right-hand side of (\ref{star}) to be finite.
We choose $\alpha_j = 2^{j/2} \cdot 8\beta R_0^{1/4}$. Recall $C_0(\alpha)=c\nu\alpha^2 L+ c\alpha^3 L^3+ c\beta \alpha^2 L^3+ \frac{c\gamma^2}{\nu^4} \alpha^5 L^5$.
Then, $C_0(\alpha_j)=c(2^j \nu  \beta^2 R_0^{1/2}  L+ 2^{3j/2} \beta^3 R_0^{3/4} L^3 + 2^j \beta^3 R_0^{1/2} L^3+\frac{\gamma^2}{\nu^4}2^{5j/2}\beta^5 R_0^{5/4} L^5).$
In order to match the requirement
$4^{j-1} R_0^2 \alpha_j  \geq \frac{4}{\nu}|A^{-\frac{1}{2}} f |^2  +C_0(\alpha_j)$ with $\alpha_j=2^{j/2} \cdot 8\beta R_0^{1/4}$, for all $j=0, 1, 2, \cdots$, i.e.,
\begin{align*}
R_0^{\frac{9}{4}} \geq \frac{2}{\nu \beta}|A^{-\frac{1}{2}}f|^2 \cdot 2^{-\frac{5}{2}j}+
c(2^{-\frac{3}{2}j}\nu \beta R_0^{\frac{1}{2}} L+  2^{-j}\beta^2 R_0^{\frac{3}{4}} L^3 +2^{-\frac{3}{2}j}\beta^2 R_0^{\frac{1}{2}} L^3 + \frac{\gamma^2}{\nu^4} \beta^4 R_0^{\frac{5}{4}} L^5 ),
\end{align*}
for all $j=0, 1, 2 , \cdots$, a sufficient condition is
\begin{align} \label{con-R}
R_0 \geq \frac{2}{\nu\beta}|A^{-\frac{1}{2}} f |^2+c\cdot \max\left\{\nu \beta L, \;  \beta^2  L^3 ,\;  \frac{\gamma^2}{\nu^4}\beta^4 L^5\right\}+1.
\end{align}

With $\alpha_j=2^{j/2}\cdot 8\beta  R_0^{1/4}$ and $R_0$ satisfying (\ref{con-R}), we obtain from (\ref{star}) that
\begin{align*}
\sum_{j=0}^{\infty} (t_{j}-t_{j+1})
\leq \frac{15}{4}\sum_{j=0}^{\infty}  \frac{1}{2^{j/2}\beta R_0^{1/4}}
+\sum_{j=0}^{\infty}\frac{c  2^{j/2} \beta R_0^{1/4} L^3}{4^{j} R_0^2}=C(\frac{1}{\beta R_0^{1/4}}+\frac{\beta L^3}{R_0^{7/4}}).
\end{align*}
This can be interpreted as: if $|u(t_0)|\geq R_0$, and let the time $t$ goes backwards, then the lifespan of the solution backward in time from $t=t_0$  is shorter than  $C(\frac{1}{\beta R_0^{1/4}}+\frac{\beta L^3}{R_0^{7/4}}) $.
Therefore, there does not exist a global solution $u(t):\mathbb R \rightarrow H$ of (\ref{visKdV}) unless $u$ belongs to the global attractor.
\end{proof}

We remark that the proof of Theorem \ref{main} implies that $S(t)H$ is not dense in $H$ for any $t>0$, where $S(t)$ is the solution semigroup of the KBS equation (\ref{visKdV}) with $\nu>0$. This fact is contrast to the Bardos-Tartar conjecture on the 2D NSE \cite{Bardos-Tartar}. More precisely, we have the following result.

\begin{corollary} \label{corr}
Assume $\nu>0$ and $f\in V'$. For any $t'>0$, there exist $r(t')>0$ such that $|S(t)u_0|< r(t')$, for all $u_0\in H$, and all $t\geq t'$, where $S(t)$ is the solution semigroup of the KBS equation (\ref{visKdV}).
\end{corollary}

\begin{proof}
The result can be directly obtained by the following statement in the proof of Theorem \ref{main}:  if $u(t)$ is a solution of (\ref{visKdV}) with
$|u(t_0)|\geq R_0$, then the lifespan of this solution backward in time starting at $t=t_0$ is shorter than $C(\frac{1}{\beta R_0^{1/4}}+\frac{\beta L^3}{R_0^{7/4}})$. Now, given any $t'>0$, we choose $r(t')$ sufficiently large so that $C(\frac{1}{\beta r(t')^{1/4}}+\frac{\beta L^3}{r(t')^{7/4}}) < t'$. Then, for any $u_0\in H$, the energy $|S(t)u_0|< r(t')$ for all $t\geq t'$, i.e., the trajectories of all solutions stay in a ball of radius $r(t')$ in $H$ for all $t\geq t'$.
\end{proof}

\bigskip

\bigskip

\section{Additional Examples}   \label{sec-dis}
In this section we discuss the backward behavior of another three dissipative evolution equations, in order to shed light on different types of backward blow-up mechanism. We also mention some open problems as well.

\subsection{Damped driven nonlinear Schr\"odinger equation}  \label{DNLS}
In this subsection we consider the 1D damped driven nonlinear Schr\"odinger equation on torus (circle)
\begin{align} \label{dNLS}
i u_t+u_{xx}+|u|^2 u + i \lambda u =f , \;\; x\in \mathbb T,
\end{align}
where $\lambda>0$, so that $i \lambda u$ is a weak damping.
The conservation form of this equation ($\lambda=0$ and $f=0$)
has been extensively studied as a fundamental equation
in modern mathematical physics (see, e.g. \cite{Bou-1,Zakharov}). For the case $\lambda>0$ and $f\in L^2_{per}(\mathbb T)$, it is shown in \cite{Ghi-NLS}  that (\ref{dNLS}) is globally well-posed in $H^1_{per}(\mathbb T)$ for $t\in \mathbb R$ and possesses a weak attractor in $H^1_{per}(\mathbb T)$, and this result has been improved in \cite{Wang-NLS} that the weak attractor is in fact a global attractor $\mathcal A$ in the usual strong topology sense, and that it has finite Hausdorff and fractal dimensions. Also, it is shown in \cite{Gou-NLS} that the global attractor $\mathcal A$ of (\ref{dNLS}) is smooth (i.e. $C^{\infty}$) provided the forcing term is smooth. Furthermore, it is proved in \cite{Oliver-Titi}, by employing the Gevrey class technique, that $\mathcal A$ is in fact contained in a subclass of the real analytic functions provided the driving term $f$ is real analytic. Finally, a recent work \cite{Jolly-Sadigov-Titi} shows that $\mathcal A$ is embedded in the long-time dynamics of a determining form (an ODE), and there is a one-to-one identification with the trajectories in $\mathcal A$ and the steady states of the determining form.

The backward behavior of the solution $u(t)$ of (\ref{dNLS}), with $u_0\in H^1_{per}(\mathbb T)$, can be described as follows: if $u(t)$ does not belong to the global attractor, then as $t\rightarrow -\infty$, the $L^2$-norm of $u(t)$ grows exponentially fast, while the $H^1$-norm grows at most exponentially in time. In particular, we have the following:
\begin{theorem} \label{thmNLS}
Assume $f\in L^2_{per}(\mathbb T)$. Let $u(t)$ be a global solution of (\ref{dNLS}) in $H^1_{per}(\mathbb T)$ for all $t\in \mathbb R$, with the initial data $u_0\in H^1_{per}(\mathbb T)$, such that $u$ does not belong to the global attractor. Then
\begin{align}
c_1 e^{-\frac{1}{2}\lambda t}   \leq     |u(t)|   \leq   C_1 e^{-\frac{3}{2}\lambda t}, \text{\;\;as\;\;}  t\rightarrow -\infty,
\end{align}
where $c_1$ and $C_1$ depends on $|u_0|$, $|f|$ and $\lambda$. In addition, the behavior of the energy $E(t)$ backward in time satisfies
\begin{align*}
E(t):=|u_x(t)|^2 - \frac{1}{2} |u(t)|_{L^4}^4  \leq C_2 e^{-3\lambda t}, \text{\;\;as\;\;} t\rightarrow -\infty,
\end{align*}
where $C_2$ depends on $\|u_0\|$, $|f|$ and $L$. Furthermore,
\begin{align*}
\|u(t)\| \leq C_3 e^{-\frac{9}{2}\lambda t}, \text{\;\;as\;\;} t\rightarrow -\infty,
\end{align*}
where $C_3$ depends on $|u_0|$, $|f|$ and $\lambda$.
\end{theorem}

\begin{proof}
See the appendix.
\end{proof}

Finally, we point out that our purpose of discussing equation (\ref{DNLS}) is to introduce the backward behavior of solutions for a typical nonlinear dispersive equation with a linear damping: the solution can be extended to all $t\in \mathbb R$ and it grows exponentially backward in time.

\smallskip

\subsection{Complex Ginzburg-Landau equation}
In this subsection we consider the complex Ginzburg-Landau equation:
\begin{align}  \label{GL}
u_t-(a+bi)u_{xx} - \delta u + (\alpha+\beta i) |u|^2 u =0,   \text{\;\;in\;\;} \mathbb T,
\end{align}
where $a\geq 0$, and $b,\delta,\alpha,\beta$ are all real-valued. Roughly speaking, this is a strongly dissipative version of (\ref{dNLS}).

First of all, we notice that, if $a=\alpha=0$, then the equation (\ref{GL}) reduces to a nonlinear Schr\"odinger equation with a linear perturbation, which is globally well-posed for  $t\in \mathbb R$. We have discussed this case in section \ref{DNLS}.

Next, we study a typical parameter regime for equation (\ref{GL}): $a> 0$, $\alpha>0$. In this case, equation (\ref{GL}) is globally well-posed forward in time and has nontrivial asymptotic behavior, which has been extensively studied in the literature (see, e.g. \cite{Aranson-02,Doering-88,Kukavica-99,Moon-82} and references therein).
The backward behavior of the solution to (\ref{GL}) in this parameter regime can be investigated straightforwardly. Indeed, we have the following:
\begin{theorem} \label{thm-GL}
Assume $a>0$ and $\alpha>0$. Let $u(t): [0,\infty) \rightarrow L^2(\mathbb T)$ be a solution of (\ref{GL}) which does not belong to the global attractor.
Then $u(t)$ cannot be extended to a global solution of  (\ref{GL}) for all $t\in \mathbb R$, i.e., any extension of $u(t)$ backward in time blows up in finite time.
\end{theorem}

\begin{proof}
We multiply (\ref{GL}) with $\bar u$ and integrate over $\mathbb T$, and then take the real part. It follows that
\begin{align}    \label{en-GL}
\frac{1}{2}\frac{d}{dt} |u|^2 + a |u_x|^2  - \delta |u|^2 +\alpha |u|_{L^4}^4=0.
\end{align}
Since $a\geq 0$, we obtain
\begin{align*}
\frac{1}{2}\frac{d}{dt} |u|^2 \leq \delta |u|^2 -\alpha |u|_{L^4}^4 \leq  \delta |u|^2 - \frac{\alpha}{L}|u|^4,
\end{align*}
by Cauchy-Schwarz inequality. By setting $s=-t$ and $y=|u|^2$, one has
\begin{align*}
\frac{dy}{ds}  \geq - 2 \delta y + \frac{2 \alpha}{L}  y^2.
\end{align*}
Hence, if $y(s_0) > \frac{\delta L}{\alpha}$, then $y(s)$ blows up at some finite time $s=s_1>s_0$.
\end{proof}

Theorem \ref{thm-GL} states that any solution $u$ to (\ref{GL}) off the global attractor blows up backward in finite time with respect to the $L^2(\mathbb T)$ norm, provided $a>0$, $\alpha>0$. We stress here, since we drop the viscosity in the above estimate, the backward blow-up in the case is solely due to the nonlinearity, and it does not result from the presence of the viscosity. This is a totally different mechanism, contrast to the backward blow-up phenomenon of the KBS equation (\ref{visKdV}), which is caused by two forces: the cascade of energy to small scales by the nonlinearity $uu_x$, as well as the amplification of higher modes by the viscosity.
In a word, our purpose for discussing equation (\ref{GL}) is to point out a backward blow-up mechanism for certain dissipative PDEs, which is completely due to the effect of the nonlinearity.

Another typical parameter regime for (\ref{GL}) is that $a>0$, $\alpha<0$. Under this scenario, equation (\ref{GL}) is locally well-posed forward in time, and by referring to the energy identity (\ref{en-GL}), we see that the nonlinear term acts as a source forward in time, which can lead the solution to blow up in finite time \cite{Ball-blowup} (see also \cite{Larios-Titi-10,Larios-Titi-14} for other results concerning finite time blow-up of similar equations). On the other hand, for the CGLE defined in the whole space $\mathbb R^{d}$, the blow-up phenomenon has been carefully discussed in  \cite{Cazenave-13,Masmoudi-08,Rot-08}, etc. However, since $\alpha<0$, we see that the nonlinearity tends to damp the energy as the time goes backwards, thus the life span of solutions backward in time might be a subtle problem, which deserves a future study.

\smallskip

\subsection{Hyperviscous Navier-Stokes equations} \label{hyper}
We have mentioned in the introduction that, in studying the backward behavior of the 2D periodic NSE \cite{backward-NSE}, a pair of orthogonality relations of the form $((u\cdot \nabla)v,v)=0$ and $((u\cdot \nabla)u, \Delta u)=0$ is critical. Also, similar orthogonality relations are valid for the 2D periodic Navier-Stokes-$\alpha$ model \cite{Vukadinovic-02}. However, consider the 2D periodic hyperviscous NSE
\begin{align} \label{NSE-h}
\begin{cases}
u_t+\nu \Delta ^2 u + (u \cdot \nabla) u+\nabla p=f,   \;\; x\in \mathbb T^2,\\
\text{div\;} u=0, \\
u(0)=u_0,
\end{cases}
\end{align}
with mean value zero assumption on $u_0$, $f$ and $u$.
Notice that the scalar product of the nonlinear and linear terms of (\ref{NSE-h}), i.e. $((u\cdot \nabla)u,\Delta^2 u)$, does not necessarily vanish, so the technique presented in \cite{backward-NSE} seems not directly applicable to (\ref{NSE-h}). Thus, the backward behavior of equation (\ref{NSE-h}) is an interesting problem to investigate.

One may first consider a simpler case $f=0$. In the absence of external forces, the energy satisfies $\frac{1}{2}\frac{d}{dt}|u|^2+\nu |\Delta u|^2=0$, then by Poincar\'e inequality, we see that the energy $|u(t)|$ decays to zero at least exponentially in time, as $t\rightarrow \infty$, and the system has a trivial attractor. In fact, one can show a stronger result: the decay forward in time is exactly of exponential type, i.e., $|u(t)|$ cannot decrease faster than exponentially as $t\rightarrow \infty.$
Concerning the 2D or 3D NSE on bounded or periodic domain without forcing, the same exponential decay result is also valid, which was proved by Foias and Saut in \cite{Foias-Saut}, and more precisely they showed that the ratio of the enstrophy over the energy converges to an eigenvalue of the Stokes operator as $t\rightarrow \infty$. A similar result also holds for the solution of (\ref{NSE-h}) with $f=0$, i.e., the ratio $|\Delta u(t)|^2/|u(t)|^2$ is convergent to an eigenvalue of the bi-Laplacian as $t\rightarrow \infty$.

In order to state the following result, we define the space
$\mathcal H:=$closure of $\mathcal V$ in $L^2_{per}(\mathbb T^2)$ where
$\mathcal V=\{\varphi=\text{trigonometric polynomials with values in} \;\mathbb R^2: \text{\;div\;} \varphi=0, \int_{\mathbb T^2} \varphi dx =0  \}$. The well-posedness of (\ref{NSE-h}) for $t\in [0,\infty)$ with $u_0\in \mathcal H$ is classical.

\begin{theorem} \label{thm-NSE}
Let $f=0$ and $u_0\in \mathcal H$, then the solution $u$ of (\ref{NSE-h}) satisfies
\begin{align}  \label{swi}
e^{-b(t-t_0)}  |u(t_0)| \leq   |u(t)|   \leq  e^{-\nu \lambda_1^2 (t-t_0)}|u(t_0)|, \text{\;\;for\;\;} t\geq t_0>0,
\end{align}
where $b=\nu  \frac{|Au(t_0)|^2}{|u(t_0)|^2}    \exp{\left\{c \nu^{-\frac{3}{2}} \lambda_1^{-2}   |u(t_0)|^2 \right \} }         $
and $\lambda_1=(2\pi/L)^2$.
\end{theorem}

\begin{proof}
The proof follows the approach in \cite{Foias-Saut}. For the sake of completion, we provide the proof in the appendix.
\end{proof}

\begin{remark}
Assume that there exists a global solution $u(t)$, $t\in \mathbb R$, of (\ref{NSE-h}) with $f=0$. By switch $t$ and $t_0$ in (\ref{swi}), we obtain
\begin{align*}
|u(t)|\geq e^{\nu \lambda_1^2 (t_0-t)}|u(t_0)|, \text{\;\;for\;\;}  -\infty < t \leq t_0 <\infty.
\end{align*}
This shows that the energy $|u(t)|$ of a nonzero global solution $u$ grows \emph{at least} exponentially fast as the time $t$ goes backwards. Thus, there are three possibilities of the backward-in-time behaviors of the solution $u$, that is, $|u(t)|$ increases exactly of exponential type as $t\rightarrow -\infty$, or faster than exponential, or blow up backward in finite time. Here, concerning the 2D periodic hyperviscous NSE (\ref{NSE-h}), we state a similar conjecture to the one of Bardos-Tartar \cite{Bardos-Tartar}: $S(t) \mathcal H$ is dense in $\mathcal H$ for every $t>0$, where $S(t)$ is the solution semigroup of (\ref{NSE-h}). In fact, we can propose a slightly stronger conjecture for (\ref{NSE-h}): the set of all $u_0\in \mathcal H$ for which $S(t)u_0$ is a global solution for $t\in \mathbb R$ is dense in $\mathcal H$.

Finally, we point out that, for the 2D periodic hypervicous NSE (\ref{NSE-h}) without forcing, i.e., $f=0$, there exists a class of initial data $u_0$ for which the solution $S(t)u_0$ of (\ref{NSE-h}) is global, for all $t\in \mathbb R$, such that the energy $|S(t)u_0|$ grows exponentially fast as $t\rightarrow -\infty$. In order to construct these explicit solutions we write (\ref{NSE-h}) with $f=0$ in the equivalent vorticity-stream formulation (see, e.g. \cite{Majda-Bertozzi})
\begin{align} \label{v-s}
\omega_t+\nu \Delta^2 \omega+J(\psi, \Delta \psi)=0, \;\; x\in \mathbb T^2,
\end{align}
where the vorticity $\omega:=\text{curl}\, u$, $\psi$ is the stream function satisfying
$u=\nabla^\perp \psi:=(-\psi_{x_2},\psi_{x_1})^{tr}$, and the Jacobian
$J(\psi, \Delta \psi):= \begin{vmatrix} \psi_{x_1} & \psi_{x_2} \\ \Delta \psi_{x_1}   & \Delta \psi_{x_2} \end{vmatrix} $.
Let us define, for any fixed $\beta \in \mathbb N$,
$$\psi_0(x)=\sum_{\mathbf k\in \mathbb Z^2, \,|\mathbf k|^2=\beta}  [a_{\mathbf k} \cos(\frac{2\pi \mathbf k}{L} \cdot x )+ b_{\mathbf k} \sin(\frac{2\pi \mathbf k}{L} \cdot x )],  \;\;x\in \mathbb T^2, $$
where $a_{\mathbf k}$ and $b_{\mathbf k}$ are arbitrary real numbers. Notice that $\psi_0$ is a steady state of the 2D Euler equation $\omega_t+J(\psi, \Delta \psi)=0$ on $\mathbb T^2$, due to the fact $J(\psi_0,\Delta \psi_0)=0$ (see, e.g. \cite{Majda-Bertozzi}). Inspired by these steady state solutions of the Euler equation, we see that $$\psi(x,t)=e^{-16 \pi^4 \beta^2 \nu t/L^4} \psi_0(x)$$ is an explicit solution of (\ref{v-s}), since for this particular stream function, one has $J(\psi, \Delta \psi)=0$, and the vorticity $$\omega(x,t)=\Delta \psi(x,t)=\frac{4\pi \beta}{L^2} e^{-16 \pi^4 \beta^2 \nu t /L^4}\psi_0(x).$$
Also, we can calculate the velocity
\begin{align*}
u_0=\nabla^{\perp} \psi_0=
\begin{pmatrix}
\sum_{\mathbf k=(k_1,k_2)\in \mathbb Z^2, \,|\mathbf k|^2=\beta}  [a_{\mathbf k} k_2 \sin(\frac{2\pi \mathbf k}{L} \cdot x )- b_{\mathbf k} k_2 \cos(\frac{2\pi \mathbf k}{L} \cdot x )] \\
\sum_{\mathbf k=(k_1,k_2)\in \mathbb Z^2, \,|\mathbf k|^2=\beta}  [-a_{\mathbf k} k_1 \sin(\frac{2\pi \mathbf k}{L} \cdot x )+ b_{\mathbf k} k_1 \cos(\frac{2\pi \mathbf k}{L} \cdot x )]
\end{pmatrix}.
\end{align*}
Since (\ref{v-s}) is equivalent to (\ref{NSE-h}) with $f=0$, we have that $$u(x,t)=\nabla^{\perp}\psi=e^{-16 \pi^4 \beta^2 \nu t/L^4} u_0$$
is an explicit global solution of (\ref{NSE-h}) with $f=0$, for all $t\in \mathbb R$. Obviously, the $L^2$-norm $|u(t)|=e^{-16 \pi^4 \beta^2 \nu t/L^4} |u_0|$ grows exponentially fast backward in time. In other words, there exists a family of initial data $u_0$, constructed above, for which the solution $S(t)u_0$ of (\ref{NSE-h}) is global, for all $t\in \mathbb R$, with the $L^2$-norm growing exponentially fast backward in time. However, it is an open problem whether that the set of initial data, for which the solution of (\ref{NSE-h}) can be extended to all $t\in \mathbb R$, is dense in $\mathcal H$.

\end{remark}

\bigskip

\section{Discussion}   \label{spectra}
In this section we would like to provide some discussion for the purpose of understanding the different backward behaviors of the 2D periodic NSE and the KBS equation (\ref{visKdV}). In section \ref{soliton}, we will consider a special case of the KBS equation: the viscous KdV equation, and demonstrate that a cnoidal-like wave solution (periodic soliton) blows up backward in time. In section \ref{en-spec} we employ the energy spectra and argue, using Kolmogorov turbulence theory, in an attempt to explain why the solutions of the 2D periodic NSE can be extended to a global solution, for all $t\in \mathbb R$, for a rich set of initial data, but any solution outside the global attractor of the KBS equation (\ref{visKdV}), with $\nu>0$, must blow up backward in finite time. Our discussion is intended to shed more light on the nature of this phenomenon using physical non-rigorous arguments.

\subsection{Physical interpretation of backward behaviors of perturbed KdV} \label{soliton}
We aim to discuss the backward blow-up of the KBS equation (\ref{visKdV}) using perturbation arguments on explicit solutions. For the sake of simplicity, we set $\beta=0$, $f=0$ and $\nu>0$ in (\ref{visKdV}), the equation becomes a viscous perturbation of the KdV equation
\begin{align}   \label{vkdv}
u_t-\nu u_{xx} + u u_x + \gamma u_{xxx}=0,  \text{\;\;in\;\;} \mathbb T=[-L/2,L/2],
\end{align}
with the spatial mean value zero assumption on $u$. We have shown that all of the nonzero solutions of (\ref{vkdv}) blow up backward in finite time. It is well-known that the existence of the soliton solutions for KdV represents a balance between the dispersion and nonlinearity. Adding a viscous term to the KdV equation affects such balance, the viscosity dissipates the kinetic energy and tends to diminish the soliton forward in time. However, backward in time, the viscosity acts as a source that may intensify the soliton, and cause the deformation of its shape, leading to higher amplitude and shrinking width. Thus the blow-up of the soliton backward in time may result from the amplitude of the perturbed soliton approaching infinity.

To demonstrate this argument more clearly, we shall look at an explicit periodic soliton solution of the KdV equation and employ a standard perturbation argument. Let us consider the classical KdV equation
\begin{align}  \label{KdV-T}
u_t+uu_x+u_{xxx}=0, \text{\;\;in\;\;} \mathbb T.
\end{align}
It was discovered by Gardner, Greene, Kruskal, and Miura in their seminal work \cite{Gardner-67,Gardner-74} that, if $u$ is a solution of the KdV (\ref{KdV-T}), then the eigenvalues of the Schr\"odinger operator $\partial_{xx}+\frac{1}{6}u$ are independent of time. Specifically, suppose that there exist $\lambda(t)$ and a function $\psi(x,t)\in H^2(\mathbb T)$, with $\int_{\mathbb T} \psi^2(x) dx=1$, satisfying
\begin{align}  \label{Sturm}
(\partial_{xx}+\frac{1}{6}u)\psi=\lambda \psi,    \text{\;\;in\;\;} \mathbb T,
\end{align}
then $\lambda(t)$ is a constant in time. This can be seen readily by a formal argument \cite{Gardner-67}. Indeed, a straightforward manipulation of (\ref{Sturm}) implies
\begin{align}   \label{soli-1}
\frac{\partial}{\partial x} (\psi_x R - \psi R_x )=-\psi^2 (\lambda_t - \frac{1}{6}(  u_t +  u u_x + u_{xxx} )  ),
\end{align}
where $R(x,t):=\psi_t - \frac{1}{6} u_x \psi + (\frac{1}{3}u+4\lambda) \psi_x$. Since $u$ is a solution of (\ref{KdV-T}), then (\ref{soli-1}) is reduced to
\begin{align}   \label{soli-2}
\frac{\partial}{\partial x} (\psi_x R - \psi R_x )=-\psi^2 \lambda_t.
\end{align}
Integrating (\ref{soli-2}) in $\mathbb T$, by using the periodicity of $\psi$ and $u$, as well as the square integrability of $\psi$, we infer that $\lambda_t=0$, i.e., $\lambda(t)$ is \emph{invariant} in time.

By employing the technique mentioned above, a perturbation of the KdV equation can be analyzed. The method is standard and can be found in \cite{Ablowitz, Drazin, Kaup-Newell}. Indeed, let us consider a perturbed KdV:
\begin{align}   \label{vari-KdV}
\tilde u_t+\tilde u \tilde u_x + \tilde u_{xxx}= \epsilon q,   \text{\;\;in\;\;} \mathbb T,
\end{align}
where $\epsilon>0$ is small and $q(x,t)$ may depend on $\tilde u$. The goal is to find the evolution of the eigenvalues of the operator $\partial_{xx}+\frac{1}{6} \tilde u$
provided $\tilde u$ is a solution of (\ref{vari-KdV}). To this end, we consider the equation
\begin{align}  \label{sturm'}
(\partial_{xx}+\frac{1}{6}\tilde u)\tilde \psi=\tilde \lambda \tilde \psi, \text{\;\;in\;\;} \mathbb T, \text{\;with\;} \int_{\mathbb T} \tilde \psi^2(x) dx=1.
\end{align}
Exactly the same as (\ref{soli-1}), we have
\begin{align}   \label{soli-1'}
\frac{\partial}{\partial x} (\tilde \psi_x \tilde R - \tilde \psi \tilde R_x )=-\tilde \psi^2 (\tilde \lambda_t - \frac{1}{6}(  \tilde u_t +  \tilde u \tilde u_x + \tilde u_{xxx} )  ),
\end{align}
where we define $\tilde R(x,t)=\tilde \psi_t - \frac{1}{6} \tilde u_x \tilde \psi + (\frac{1}{3}\tilde u+4\tilde \lambda) \tilde \psi_x$.
By substituting (\ref{vari-KdV}) into (\ref{soli-1'}), it follows that
\begin{align}    \label{soli-3}
\frac{\partial}{\partial x} (\tilde \psi_x \tilde R - \tilde \psi \tilde R_x )=\tilde \psi^2 (\tilde \lambda_t - \frac{1}{6}\epsilon q  ).
\end{align}
Unlike the situation of the classical KdV equation (\ref{KdV-T}), here an eigenvalue $\tilde \lambda(t)$ is no longer invariant in time. Indeed, by integrating (\ref{soli-3}) over $\mathbb T$ and using the periodicity of $\tilde \psi$ and $\tilde u$, it follows that
\begin{align}   \label{soli-4'}
\tilde \lambda_t \int_{\mathbb T} \tilde \psi^2 dx  = \frac{1}{6} \epsilon \int_{\mathbb T} \tilde \psi ^2 q dx.
\end{align}
Since $\int_{\mathbb T} \tilde \psi^2 dx = 1$, then
\begin{align}   \label{soli-4}
\tilde \lambda_t   = \frac{1}{6} \epsilon \int_{\mathbb T} \tilde \psi ^2 q dx.
\end{align}
Note that (\ref{soli-4}) shows that an eigenvalue $\tilde \lambda(t)$ may evolve slowly within a short time provided $\epsilon$ is small.

Next we look at a typical explicit periodic solution of the KdV equation defined in the torus $\mathbb T=[-\frac{L}{2},\frac{L}{2}]$. Take the initial data to be $u_0=12m_0 l_0^2 \text{\,cn}^2 (l_0 x,m_0)$,
where $m_0\in (0,1)$ is the elliptic parameter for the Jacobi elliptic function $\text{cn}(x,m_0)$. Here, we require $m_0$ and $l_0$ to satisfy the relation
\begin{align} \label{ml}
l_0=\frac{1}{L}\int_0^{2\pi} \frac{1}{\sqrt{1-m_0 \sin^2\theta}} d\theta
\end{align}
so that the period of $u_0(x)$ is $L$, due to the definition $\text{cn}(x,m_0)=\cos(\phi)$ where $x=\int_0^{\phi} \frac{1}{\sqrt{1-m_0\sin^2 \theta}} d\theta$.
With this initial data $u_0$, the unique solution of the KdV equation (\ref{KdV-T}) is
\begin{align} \label{cnoidal}
u(x,t)=12 m_0 l_0^2 \text{\,cn}^2 (l_0(x-c_0 t),m_0),
\end{align}
where $c_0=4(2m_0-1)l_0^2$ (see, e.g., \cite{Drazin,Liu}). This is the so-called ``cnoidal" wave solution (periodic soliton) derived by Korteweg and de Vries in their 1895 paper \cite{KdV} in which they proposed the KdV equation.
Concerning the spectral problem (\ref{Sturm}) with the cnoidal wave solution $u$ given in (\ref{cnoidal}), we find an eigenvalue
\begin{align}  \label{ei-v}
\lambda=(2m_0-1)l_0^2,
\end{align}
with the corresponding eigenfunction
\begin{align}    \label{ei-f}
\psi(x,t)=\sqrt{\frac{l_0}{C_{m_0}}} \text{\,cn}(l_0(x-c_0 t),m_0),
\end{align}
where $C_{m_0}= \int_{-K_{m_0}/2}^{K_{m_0}/2} \text{\,cn}^2(x,m_0) dx $, where $K_{m_0}=\int_0^{2\pi} \frac{1}{\sqrt{1-m_0 \sin^2\theta}} d\theta$ is the period of the elliptic function $\text{cn}(x,m_0)$. With this definition of $\psi$, one has $\int_{\mathbb T} \psi^2(x) dx=1$.

For the same initial data $u_0=12 m_0 l_0^2 \text{\,cn}^2 (l_0 x,m_0)$, we now investigate the solution $\tilde u$ of the perturbed KdV equation (\ref{vari-KdV}) provided $\tilde u(x,0)=u_0$. Since $\epsilon>0$ is very small, we expect the solution $\tilde u$ of (\ref{vari-KdV}) to be a small perturbation of the KdV solution (\ref{cnoidal}), at least for a short time, in the sense that,
\begin{align}   \label{cnoidal'}
\tilde u(x,t) \approx 12 m l^2 \text{\,cn}^2(l(x-ct),m),
\end{align}
where $l(t)=\frac{1}{L}\int_0^{2\pi} \frac{1}{\sqrt{1-m(t) \sin^2\theta}} d\theta$ and $c(t)=4(2m(t)-1)l^2(t)$, with $0<m(t)<1$,
such that $l(0)=l_0$ and $m(0)=m_0$. As a result, for the spectral problem (\ref{sturm'}), it is expected that there exists an eigenvalue $\tilde \lambda$ with the corresponding eigenfunction $\tilde \psi$, which are small perturbation of $\lambda$ and $\psi$, defined in (\ref{ei-v}) and (\ref{ei-f}), respectively, i.e.,
\begin{align}  \label{ei-v'}
\tilde \lambda(t)\approx(2m(t)-1)l^2(t),
\end{align}
with the corresponding eigenfunction
\begin{align}    \label{ei-f'}
\tilde \psi(x,t) \approx   \sqrt{\frac{l}{C_{m}}} \text{\,cn}(l(x-ct),m),
\end{align}
where $C_{m}(t)= \int_{-K_{m}/2}^{K_{m}/2} \text{\,cn}^2(x,m(t)) dx $, where $K_{m}(t)=\int_0^{2\pi} \frac{1}{\sqrt{1-m(t) \sin^2\theta}} d\theta$ is the period of the elliptic function $\text{cn}(x,m(t))$.

In order to find out how the shape of the cnoidal-like wave solution $\tilde u$ changes in time for the perturbed KdV equation (\ref{vari-KdV}) with small perturbation $\epsilon q$, we study some typical perturbations.

\subsubsection{The damped KdV equation} First set $q=-\tilde u$ in (\ref{vari-KdV}), i.e., we consider the KdV equation with a small linear damping
\begin{align} \label{soli-8}
\tilde u_t+\tilde u \tilde u_x + \tilde u_{xxx}= -\epsilon \tilde u, \text{\;\;in\;\;} \mathbb T.
\end{align}
Equation (\ref{soli-8}) is a special case of a more general equation $u_t+uu_x + u_{xxx}= -\Gamma(t) u,$ which models a wave moving over an uneven bottom \cite{Johnson-73,Kakutani,Kaup-Newell}.
By (\ref{cnoidal'}), $q=-\tilde u \approx -12 m l^2 \text{\,cn}^2(l(x-ct),m)$, then we substitute this along with (\ref{ei-v'}) and (\ref{ei-f'}) into (\ref{soli-4}) to obtain
\begin{align}  \label{solii-1}
\partial_t[(2m-1)l^2]  \approx  \frac{-2 \epsilon ml^2}{C_{m}} \int_{-K_m/2}^{K_m/2}  \text{\,cn}^4(x,m)  dx.
\end{align}
In order to see the evolution of $l(t)$ more clearly, we consider the situation in which the elliptic parameter $m$ is very close to 1, so that $m$ is almost invariant in short time compared to the change of $l$, and then the Jacobi elliptic function $\text{cn}(x,m)$ can be approximated by the hyperbolic function $\text{sech}(x)$. Under such scenario, we obtain from (\ref{solii-1}) that
\begin{align}   \label{soli-7}
\partial_t l \approx - \epsilon\tilde C  l, \text{\;\;with\;\;} l(0)=l_0,
\end{align}
where $l_0$ is given in (\ref{ml}), and $\tilde C=\left(\int_{\mathbb R}\text{sech}^4(x)dx\right) \left(\int_{\mathbb R}\text{sech}^2(x)dx\right)^{-1}=\frac{2}{3}$.
Notice that (\ref{soli-7}) implies $l(t)\approx e^{-\epsilon \tilde C t} l_0$, i.e., forward in time, $l(t)$ decreases to zero exponentially fast. On the other hand, backward in time, $l(t)$ increases with an approximately exponential rate, i.e., the amplitude $12m l^2$ of the cnoidal wave (\ref{cnoidal'}) grows exponentially with an exponentially fast shrinking width $\sim 1/l$. Moreover, the propagation speed $c(t) = 4(2m(t)-1)l^2(t)$ of the cnoidal wave also increases exponentially backward in time.
Furthermore, by (\ref{cnoidal'}) it follows that the energy $|\tilde u(t)|^2 \sim l^3(t)$, which increases exponentially fast backward in time. Notice that this finding is consistent with the $L^2$ energy estimate of (\ref{soli-8}), i.e., $|\tilde u(t)|^2=e^{-2\epsilon t} |u_0|^2$. In addition, from (\ref{cnoidal'}) we see that $|\tilde u_x(t)|^2 \sim l^5(t)$ also tends to grow exponentially backward in time.

\subsubsection{The viscous KdV equation}   \label{sub-vkdv}
Next we set $q=\tilde u_{xx}$ in (\ref{vari-KdV}), i.e., we consider the viscous KdV equation
\begin{align} \label{soli-9}
\tilde u_t+\tilde u \tilde u_x + \tilde u_{xxx}= \epsilon \tilde u_{xx}, \text{\;\;in\;\;} \mathbb T.
\end{align}
By using (\ref{cnoidal'}), we calculate
$$q=\tilde u_{xx}\approx -24m l^4 [3m \text{\,cn}^4(l(x-ct),m)+(2-4m)\text{\,cn}^2(l(x-ct),m) +m-1  ].$$
Substituting this formula along with (\ref{ei-v'}) and (\ref{ei-f'}) into (\ref{soli-4}) gives
\begin{align} \label{solii-2}
\partial_t[(2m-1)l^2] \approx - \frac{4\epsilon m l^4}{C_m} \int_{-K_m/2}^{K_m/2}
&\text{\,cn}^2(x,m) [3 m \text{\,cn}^4(x,m)    \notag\\
&+(2-4m) \text{\,cn}^2(x,m)+m-1] dx.
\end{align}
Analogously as above, we consider the case that the elliptic parameter $m$ is very close to 1 so that $m$ is almost invariant in short time compared to the change of $l$, and then $\text{cn}(x,m)$ can be approximated by $\text{sech}(x)$. Thus,
(\ref{solii-2}) implies
\begin{align} \label{solii-3}
\partial_t l \approx -\epsilon \tilde C_1  l^3
\end{align}
where $\tilde C_1=2\left(\int_{\mathbb R} \text{sech}^4(x)[3 \,\text{sech}^2(x)-2] dx\right)
\left(\int_{\mathbb R} \text{\,\text{sech}}^2(x) dx \right)^{-1}=\frac{8}{15}$. It follows that $l(t)$ decreases to zero forward in time. In addition, we see from (\ref{solii-3}) that $l(t)$ tends to approach infinity backward in finite time, i.e., the amplitude $12ml^2$ of the cnoidal wave (\ref{cnoidal'}) grows very fast with a rapidly shrinking width $\sim 1/l$ and a fast accelerating wave speed $c=4(2m-1)l^2$, which may lead to a backward blow-up. Also, we obtain from (\ref{cnoidal'}) that the energy $|\tilde u(t)|^2 \sim l^3(t)$, tending to blow up backward in time. This finding indicates the mechanism of singularity formulation, and it is consistent with the rigorous mathematical result in Theorem 1(ii) that the energy of the viscous KdV (\ref{soli-9}) blows up backward in time.

\subsubsection{The viscous Benjamin-Bona-Mahony type of equation}
Finally, in order to test the robustness of the asymptotic argument presented in this section, we would like to study another perturbation of the KdV equation. In particular, we set $q=\tilde u_{xx}+\tilde u_{xxt}$ in (\ref{vari-KdV}), and then the equation becomes
\begin{align} \label{bbm}
\tilde u_t+\tilde u \tilde u_x +\tilde u_{xxx}=\epsilon(\tilde u_{xx}+\tilde u_{xxt}), \text{\;\;in\;\;} \mathbb T.
\end{align}
Equation (\ref{bbm}) is a viscous Benjamin-Bona-Mahony (BBM) type of equation. The original BBM equation, $u_t+u_x+uu_x-u_{xxt}=0$, was proposed as a modification of the KdV equation for modeling long surface gravity waves \cite{BBM}. The energy identity of (\ref{bbm}) reads
\begin{align*}
\frac{1}{2}\frac{d}{dt}(|\tilde u|^2+\epsilon|\tilde u_x|^2)=-\epsilon |\tilde u_x|^2.
\end{align*}
This implies, with the Poincar\'e inequality, that the energy $|\tilde u|^2+\epsilon|\tilde u_x|^2$ decreases to zero exponentially fast forward in time, while grows exponentially fast backward in time. Now we use our perturbation argument to study the evolution of the shape of the cnoidal wave for (\ref{bbm}).
Similar to arguments introduced above, we assume that the elliptic parameter $m$ is very close to 1 so that $m$ is almost invariant in short time compared to the change of $l$, and in this case, $\text{cn}(x,m)$ can be approximated by $\text{sech}(x)$. Then, after straightforward calculations, we obtain
\begin{align*}
-2 l_t \approx 2 \epsilon \tilde C_2  l^3 +32  \epsilon  \tilde C_3 l^6 + l_t(8 \epsilon \tilde C_2 l^2
-8\epsilon \tilde C_4 l^2 +64 \epsilon \tilde C_3 l^5 t ),
\end{align*}
where
\begin{align*}
&\tilde C_2=\int_{\mathbb R} \text{sech}^4(x)[3\text{\,sech}^2(x) -2  ]  dx =\frac{8}{15} \,  ;      \\
&\tilde C_3=\int_{\mathbb R} \text{sech}^4(x) \text{tanh}(x) [3\text{\,sech}^2(x) -1 ]  dx   =0  \, ;\\
&\tilde C_4=\int_{\mathbb R} \text{sech}^4(x) \text{tanh}(x) [3\text{\,sech}^2(x) -1 ] x dx =\frac{1}{5} \,.
\end{align*}
It follows that
\begin{align*}
\partial_t l \approx -\frac{8\epsilon l^3}{20 \epsilon l^2 +15}.
\end{align*}
On one hand, this implies that $l(t)$ decreases to zero forward in time. On the other hand, backward in time, for large $l$, the amplitude $12ml^2$ of the cnoidal wave (\ref{cnoidal'}) grows approximately exponentially fast with exponentially fast shrinking width $\sim 1/l$, which is consistent with exponential growth of the energy backward in time.

\smallskip

\subsection{Energy spectra} \label{en-spec}
In this section, we would like to argue that the various backward behaviors of different dissipative equations are connected to their energy spectra as well as the Kolmogorov turbulence theory. We stress that our arguments are mainly physically oriented, rather than a rigorous mathematical treatment.

\subsubsection{The viscous KdV equation}
In the previous discussion, we have given some evidence that, if we take the cnoidal wave to be the initial value of the viscous KdV equation (\ref{vkdv}), then as the time goes backwards, the amplitude $\sim l^2$ of the cnoidal wave grows fast with its width $\sim \frac{1}{l}$ shrinking rapidly, and as a result, the solution may get close to a Dirac delta function. Since the Fourier transform of a Dirac delta function is a constant, we realize that the energy spectrum $E_k^{1D}$ of the solution $u$, defined as $E_{k}^{1D}=|\hat u_k|^2$, for $k\in \mathbb Z \backslash \{0\}$, is approximately \emph{invariant} with respect to $k$, for the spatial scales $\frac{L}{|k|}$ sufficiently larger than the width $\sim \frac{1}{l}$ of the cnoidal wave (see \cite{Nazarenko}). Now, by setting $s=-t$, we consider the energy identity of the viscous KdV (\ref{vkdv}), backward in time,
\begin{align}     \label{nergy}
\frac{1}{2}\frac{d}{ds}|u(s)|^2 = \nu |u_x(s)|^2 &=  \nu \sum_{k \in \mathbb Z \backslash \{0\}   } k^2 E_k^{1D}(s)  \notag\\
&\approx \nu \left(C\sum_{0<|k| \leq  k_{max}(s)} k^2+ \sum_{|k| >  k_{max}(s)} k^2 E_k^{1D}(s)\right) .
\end{align}
Notice that $k_{max}(s) \sim l(s)L$, where $L$ is the length of the domain $\mathbb T$. Recall we have argued in subsection \ref{sub-vkdv} that $l(s)$ tends to blow up in finite time, thus $k_{max}(s)$ increases very fast, which leads to a fast growth of $\frac{d}{ds}|u(s)|^2$ due to (\ref{nergy}), ending up with an energy blow-up in finite time.

\subsubsection{The Burgers equation}
The KBS equation (\ref{visKdV}) is also related to the Burgers equation. In fact, by setting $\beta=\gamma=0$ and $f=0$ in (\ref{visKdV}), it reduces to the Burgers equation $u_t-\nu u_{xx}+u u_x = 0$, and we have proved that all of its nonzero solutions blow up backward in finite time. It is well-known that the Burgers spectrum behaves like $E_k^{1D} \sim k^{-2}$ (see \cite{Burgers}). An energy formula analogous to (\ref{nergy}) indicates the energy tends to grow very fast as the time goes backwards, leading to a finite time blow-up.

\subsubsection{The 2D Navier-Stokes equations}
Next we attempt to discuss the relation between the Kolmogorov turbulence theory and the backward behavior of 2D NSE. The Richardson energy cascade theory \cite{Richardson} states that the energy transfers among eddies with similar sizes only, and the rate of the energy injection at large scales is equal on average to the energy dissipation rate at small scales, so that a statistically steady turbulent state forms. Far away from the source and sink, there is an inertial range, in which the turbulence properties only depend on the energy cascade rate. Based on a simple dimensional argument, in 1941, Kolmogorov and Obukhov \cite{Kolmogorov-41,Kolmogorov-41',Obukhov} derived a celebrated result of the energy spectrum for the 3D turbulence $\mathcal E_{\kappa} \sim C_{\epsilon} \epsilon^{\frac{2}{3}} \kappa^{-\frac{5}{3}}$ provided the wavenumber $\kappa \geq 0$ is in an inertial range, and $\epsilon$ is the energy cascade rate (equal to the energy dissipation rate).

On the other hand, the 2D incompressible ideal flow conserves two quadratic quantities, energy and enstrophy, and thus the 2D turbulence possesses a dual cascade behavior \cite{Kraichnan}: inverse energy cascade and direct enstrophy cascade, i.e., the energy is transferred from small to large vortices, while the enstrophy is transferred to small scales. By a dimensional argument, the inverse energy cascade spectrum is identical to the 3D turbulence spectrum mentioned above, but the enstrophy cascade spectrum reads $E_{\kappa}\sim C_{\eta} \eta^{\frac{2}{3}} \kappa^{-3}$, which is called the Kraichnan spectrum \cite{Kraichnan}, where $\eta$ is the enstrophy cascade rate, and the isotropic spectrum for each $\kappa \geq 0$ is defined by
$E_{\kappa}:=2\pi \kappa |\hat u_{\kappa}|^2$, with the isotropic assumption $|\hat u_{\kappa}|=|\hat u_{\vec \kappa}|$ for all $\vec {\kappa}\in \mathbb R^2$ with $|\vec {\kappa}|=\kappa$. A more precise version of the Kraichnan spectrum with a log correction is suggested in \cite{Kraichnan'}. The Kraichnan dual-cascade picture was recently confirmed numerically in \cite{Boffetta}. For the turbulent flow in the periodic domain, the energy spectrum $E_{\kappa}$ can also be defined using the following way (see \cite{book-Foias}): it is assumed that there exists $E_{\kappa}$  such that the total energy between the wavenumbers $k'$ and $k''$, i.e. $\sum_{\mathbf k\in \mathbb Z^2, \, k'\leq |\mathbf k|<k''} |\hat u_{\mathbf k}|^2 $, can be approximated by the integral $\int_{k'}^{k''} E_{\kappa} d\kappa$, for wavenumbers $k'<k''$ in the inertial range.

Now, assume the energy is injected near wavenumber $k_f$ and dissipated at very small wavenumbers $k_{-} \ll k_f$ and at very large wavenumbers $k_+\gg k_f$, and there are neither forcing nor dissipation at wavenumbers $\kappa$ such that $k_{-}<\kappa<k_f$ or $k_f< \kappa < k_{+}$, which are the inverse and the direct cascade inertial ranges, respectively. Under such scenario, if we set $s=-t$, then the energy identity of the 2D NSE, backward in time, reads
\begin{align}   \label{Spec}
&\frac{1}{2}\frac{d}{ds}|u|^2 = \nu |\nabla u|^2
\approx \nu \int_0^{\infty} \kappa^2 E_{\kappa}   d\kappa   \notag\\
&\approx \nu \left(\int_0^{k_{-}}    \kappa^2 E_{\kappa}  d\kappa + C_{\epsilon} \epsilon^{\frac{2}{3}}\int_{k_{-}}^{k_f}   \kappa^{\frac{1}{3}} d\kappa+
C_{\eta} \eta^{\frac{2}{3}}  \int_{k_f}^{k_+}    \kappa^{-1}   d\kappa+ \int_{k_+}^{\infty}    {\kappa}^2 E_{\kappa} d\kappa  \right).
\end{align}
Since in general the lower modes are not a main source for the energy blow-up, we focus on the higher modes $\kappa>k_f$. By a simple dimensional argument, the enstrophy cascade inertial range is expected to extend up to $k_+ \sim (\eta/\nu^3)^{\frac{1}{6}}$ \cite{Kraichnan}. As a result, the third integral in (\ref{Spec}) can be estimated as follows: $C_{\eta} \eta^{\frac{2}{3}}    \int_{k_f}^{k_+}    \kappa^{-1} d\kappa  \sim C_{\eta} \eta^{\frac{2}{3}}(\frac{1}{6}\ln(\eta/\nu^3)-\ln(k_f))$, which is controlled by the enstrophy cascade rate. Concerning the last term in (\ref{Spec}) for larger wavenumbers $\kappa>k_+$, we notice that, as the time goes backwards, the viscosity amplifies the enstrophy on high wavenumbers, thus the enstrophy tends to cascade \emph{inversely} backward in time, moving from small to large scales, so that the enstrophy on the small scales might not grow too rapidly. Consequently, the energy increasing rate $\frac{d}{ds}|u(s)|^2$ is possible to be relatively slow, and the finite-time backward blow-up of 2D NSE may be prevented, for initial data with energy concentrated on low modes.

Our argument above is consistent with the result in \cite{backward-NSE}, which states that there is a rich set of initial data in the energy space $\mathcal H$ for which the solution of 2D NSE can be extended to a global solution for all $t\in \mathbb R$. More precisely, it is shown in \cite{backward-NSE} that, if $p_0\in P_n \mathcal H$, for some $n$, then there exists a global solution $S(t)u_0$ such that $P_n u_0=p_0$, with $|Q_n u_0|\leq \max\{2|f|/\nu \lambda_1, \gamma_n^{1/2}|p_0|\}$ where
$\gamma_n=(\lambda_{n+1}+\lambda_n)/(\lambda_{n+1}-\lambda_n)$, and
$$\limsup_{t\rightarrow -\infty} \frac{\|S(t)u_0\|^2}{|S(t)u_0|^2} \leq \frac{\lambda_n+\lambda_{n+1}}{2}    ,$$ where $P_n$ is the projection onto the lower modes $|k|\leq n$, $Q_n=I-P_n$, and $\lambda_n$ is the $n$th eigenvalue of the linear Stokes operator in 2D periodic NSE. This indicates that, if $|Q_n u_0|$ is controlled by $|P_n u_0|$ in an appropriate manner, then the ratio of the enstrophy to the energy can be bounded as $t\rightarrow -\infty$, which may be a result of the inverse enstrophy cascade and the direct energy cascade, \emph{backward} in time, from the Kraichnan dual cascade picture for 2D incompressible flow. Also, it is easy to see from the energy identity that, if the ratio of the enstrophy to the energy is uniformly bounded for all negative time, then the solution is global for all $t\in \mathbb R$, growing exponentially as $t\rightarrow -\infty$, provided it does not belong to the global attractor.

In sum, the energy and enstrophy cascade as well as the Kolmogorov spectra may reveal the underlying mechanism for the different backward behavior of the KBS equation (\ref{visKdV}) and the 2D NSE. In particular, we would like to emphasize that there is a huge difference between the KdV spectrum $E_k^{1D} \sim constant$ and the 2D turbulence spectrum $E_{\kappa} \sim {\kappa}^{-3}$. For more material on Kolmogorov turbulence theory, please refer to monographs \cite{book-Birnir,book-Foias,book-Frisch,book-Zakharov}.

\bigskip

\section{Appendix}
We prove Theorem \ref{thmNLS} and Theorem \ref{thm-NSE} in the appendix.

\subsection{The proof of Theorem \ref{thmNLS}}
\begin{proof}
The proof consists of two parts: the first part is devoted to the $L^2$-estimate, and the second part is devoted to estimating the energy as well as the $H^1$-norm.

(1) \emph{$L^2$-estimate}

Multiply (\ref{dNLS}) by $\bar u$ and integrate on $\mathbb T$,
\begin{align}  \label{dNLS-1}
i \int_{\mathbb T} u_t \bar u dx - |u_x|^2 + |u|^4_{L^4} + i \lambda |u|^2
= \int_{\mathbb T} f \bar u dx.
\end{align}
By taking the imaginary part of (\ref{dNLS-1}), we obtain
\begin{align}    \label{dNLS-5}
\frac{1}{2} \frac{d}{dt}  |u|^2 + \lambda |u|^2 = \text{Im}  \int_{\mathbb T} f \bar u dx
\geq -\frac{\lambda}{2} |u|^2-\frac{1}{2\lambda} |f|^2,
\end{align}
where we used Cauchy-Schwarz inequality and Young's inequality.

It follows that
\begin{align*}
\frac{d}{dt}  |u|^2 + 3\lambda |u|^2 \geq -\frac{1}{\lambda} |f|^2.
\end{align*}
By using Gronwall's inequality, we have
\begin{align}   \label{L2est}
|u(t)|^2 \leq e^{-3\lambda t}  |u_0| ^2 +  \frac{1}{3\lambda^2} |f|^2 \left( e^{-3 \lambda t} -1   \right), \text{\;\;for\;\;} t\leq 0.
\end{align}

On the other hand, (\ref{dNLS-5}) also shows
\begin{align*}
\frac{1}{2}\frac{d}{dt} |u|^2 + \lambda |u|^2 = \text{Im}  \int_{\mathbb T} f \bar u dx \leq \frac{\lambda}{2}  |u|^2 +\frac{1}{2\lambda} |f|^2.
\end{align*}
Hence,
\begin{align*}
\frac{d}{dt} |u|^2 + \lambda |u|^2 \leq \frac{1}{\lambda} |f|^2.
\end{align*}
Again, by Gronwall's inequality, we obtain
\begin{align}    \label{L2est2}
|u(t)|^2 \geq e^{-\lambda t} |u_0|^2 +  \frac{1}{\lambda^2} |f|^2 (1-e^{-\lambda t}), \text{\;\;for\;\;}  t\leq 0.
\end{align}

By (\ref{L2est}) and (\ref{L2est2}), we infer, if $u$ does not belong to the global attractor, then
\begin{align}   \label{grow2}
c_1 e^{-\frac{1}{2}\lambda t}   \leq     |u(t)|   \leq   C_1 e^{-\frac{3}{2}\lambda t}, \text{\;\;as\;\;}  t\rightarrow -\infty,
\end{align}
where $c_1$ and $C_1$ depends on $|u_0|$, $|f|$, and $\lambda$. That is to say, the $L^2$-norm of $u(t)$, which is outside the global attractor, grows exponentially fast as the time $t$ goes backwards to negative infinity.

\smallskip

(2) \emph{Estimate of the energy and the $H^1$-norm}

Multiplying (\ref{dNLS}) by $\bar u_t$ followed by integrating on $\mathbb T$, we obtain
\begin{align*}
i |u_t|^2  -\int_{\mathbb T} u_{x} \bar u_{xt} dx +\int_{\mathbb T} u^2 \bar u \bar u_t dx + i \lambda \int_{\mathbb T} u \bar u_t dx
=\int_{\mathbb T} f \bar u_t dx.
\end{align*}
Taking the real part gives
\begin{align}  \label{dNLS-2}
\frac{d}{dt} \left(|u_x|^2 - \frac{1}{2}  |u|^4_{L^4} \right)
+ 2\lambda \left(\text{Im}\int_{\mathbb T} u \bar u_t dx\right) =-2\left(\text{Re}\int_{\mathbb T} f \bar u_t dx\right).
\end{align}
Now we need to estimate $\text{Im}\int_{\mathbb T} u \bar u_t dx$. Indeed, by taking the real part of (\ref{dNLS-1}), we infer
\begin{align}  \label{dNLS-3}
\text{Im}\int_{\mathbb T} u \bar u_t  dx=-\text{Im}\int_{\mathbb T} u_t \bar u dx =
|u_x|^2- |u|^4_{L^4} + \text{Re}\int_{\mathbb T} f \bar u dx.
\end{align}
Substituting (\ref{dNLS-3}) into (\ref{dNLS-2}), we obtain
\begin{align*}
\frac{d}{dt} \left(|u_x|^2  - \frac{1}{2} |u|^4_{L^4}\right)
+2\lambda \left(|u_x|^2  -|u|^4_{L^4} + \text{Re}\int_{\mathbb T} f \bar u dx\right)
=-2\left(\text{Re}\int_{\mathbb T} f \bar u_t dx\right).
\end{align*}
This can be written as
\begin{align}   \label{dNLS-4}
\frac{d}{dt}\phi+ 2\lambda \phi = \lambda |u|^4_{L^4} + 2\lambda \left(\text{Re}\int_{\mathbb T} f \bar u dx\right),
\end{align}
where $\phi(t)=|u_x|^2  - \frac{1}{2} |u|^4_{L^4}+ 2 \left(\text{Re}\int_{\mathbb T} f \bar u dx\right)$.

Note that, $|u|^2 \leq |u|_{L^4}^2 L^{\frac{1}{2}}   \leq |u|_{L^4}^4 +   \frac{L}{4}$, by virtue of H\"older's and Young's inequalities. Thus
\begin{align}  \label{reni}
\text{Re}\int_{\mathbb T} f \bar u dx \geq -\frac{1}{2} |u|^2 -\frac{1}{2} |f|^2  \geq -\frac{1}{2} |u|_{L^4} ^4 -\frac{L}{8} -\frac{1}{2} |f|^2.
\end{align}
Applying (\ref{reni}) on (\ref{dNLS-4}), one has
\begin{align*}
\frac{d}{dt} \phi+ 2\lambda \phi \geq -\lambda\left(\frac{L}{4}+|f|^2 \right).
\end{align*}
By Gronwall's inequality, we deduce
\begin{align}   \label{estphi}
\phi(t)\leq e^{-2\lambda t} \phi(0)+\frac{1}{2} \left( \frac{L}{4}+|f|^2  \right) (e^{-2\lambda t}-1), \text{\;\;for\;\;}  t\leq 0.
\end{align}

Define the energy $E(t):=|u_x|^2  - \frac{1}{2} |u|^4_{L^4}$. Then
\begin{align*}
\phi(t) =  E(t)+ 2 \left(\text{Re}\int_{\mathbb T} f \bar u dx\right) \geq  E(t)- |f|^2 -   |u(t)|^2.
\end{align*}
Therefore, by (\ref{L2est}) and (\ref{estphi}), one has, for $t\leq 0$,
\begin{align*}
E(t) &\leq \phi(t) +|u(t)|^2 + |f|^2   \notag\\
&\leq e^{-2\lambda t} \phi(0)+\frac{1}{2} \left( \frac{L}{4}+|f|^2  \right) (e^{-2\lambda t}-1) +e^{-3\lambda t}  |u_0| ^2 +  \frac{1}{3\lambda^2} |f|^2 \left( e^{-3 \lambda t} -1   \right) + |f|^2.
\end{align*}
This shows
\begin{align}    \label{growE}
E(t)  \leq C_2 e^{-3\lambda t}, \text{\;\;as\;\;} t\rightarrow -\infty,
\end{align}
where $C_2$ depends on $\|u_0\|$, $|f|$ and $L$.

Finally, by the 1D Agmon's inequality as well as the Young's inequality, we deduce
\begin{align*}
E(t)=|u_x|^2-\frac{1}{2}|u|_{L^4}^4 &\geq |u_x|^2-c |u|^3  \|u\| \notag\\
&\geq |u_x|^2 - \frac{1}{2}\|u\|^2  -c |u|^6 \notag\\
&\geq   \frac{1}{2}\|u\|^2 - |u|^2   -c |u|^6.
\end{align*}
Thus
\begin{align*}
\|u\|^2  \leq   2E(t)+  2|u|^2  +  c|u|^6.
\end{align*}
It follows from (\ref{grow2}) and (\ref{growE}) that
\begin{align*}
\|u\| \leq C_3 e^{-\frac{9}{2} \lambda t}, \text{\;\;as\;\;} t\rightarrow -\infty,
\end{align*}
where $C_3$ depends on $|u_0|$, $|f|$ and $\lambda$.

\end{proof}

\smallskip

\subsection{The proof of Theorem \ref{thm-NSE}}
\begin{proof}
The proof follows the idea in \cite{Foias-Saut}.
Let $f=0$. Recall the space $\mathcal H$ is the space of all periodic, divergence free, $L^2$-functions on $\mathbb T^2$ with vanishing mean values. Let $P$ be the orthogonal projection on $\mathcal H$ in $L^2_{per}(\mathbb T^2)^2$ and set $Au=-P \Delta u$ and
$B(v,w)=P[(v\cdot \nabla)w]$. Then the equation (\ref{NSE-h}) can be written in an equivalent form
\begin{align}   \label{ene-NSE}
u_t + \nu A^2 u + B(u,u)=0,    \text{\;\;in\;\;}  \mathbb T^2.
\end{align}
Since $(B(u,v),v)=0$ (see, e.g. \cite{C-F-book}), taking the scalar product of (\ref{ene-NSE}) with $u$ gives
\begin{align}    \label{da}
\frac{1}{2}\frac{d}{dt}|u|^2  +  \nu |Au|^2  =0.
\end{align}
Let $\lambda_1=(2\pi/L)^2$ be the first eigenvalue of $A=-P\Delta$. Then by Poincar\'e inequality $|Au|\geq \lambda_1  |u|$, we obtain
\begin{align*}
\frac{1}{2}\frac{d}{dt}|u|^2  +  \nu \lambda_1^2  |u|^2  \leq 0.
\end{align*}
It follows that
\begin{align}    \label{dst}
|u(t)|\leq e^{-\nu \lambda_1^2 (t-t_0)}|u(t_0)|, \text{\;\;for\;\;} t\geq t_0 \geq 0.
\end{align}

Next we show that $|u(t)|\geq e^{- b (t-t_0)} |u(t_0)| ,  \text{\;\;for\;\;} t\geq t_0>0$,
where $b>0$ will be specified later.
To this end, we set
\begin{align}   \label{st}
q(t)=\frac{|Au(t)|^2}{|u(t)|^2}.
\end{align}
By differentiating both side of $q (t) |u(t)|^2=|Au(t)|^2$, we obtain
\begin{align*}
\left(\frac{d}{dt}q \right) |u|^2+2 q \left( u_t,u    \right)=
2\left(\frac{d}{dt}Au, Au \right).
\end{align*}
It follows that
\begin{align*}
\frac{1}{2} \frac{d}{dt}q  & =\frac{1}{|u|^2} (u_t,A^2 u-q u) \notag\\
&=-\frac{1}{|u|^2} (\nu A^2 u+B(u,u), A^2 u-q u) \notag\\
&=-\frac{1}{|u|^2} \nu  |A^2 u -q u|^2-\frac{1}{|u|^2} (B(u,u),A^2 u-q u)-\nu (q u, A^2 u-q u).
\end{align*}
Notice, by (\ref{st}), we have
\begin{align*}
(q u, A^2 u-q u)=\frac{|Au|^2}{|u|^2}\left(   u, A^2 u \right)-
\frac{|Au|^4}{|u|^4}  \left(u,u  \right)=0.
\end{align*}
Combining the above two estimate and by letting $v=\frac{u}{|u|}$, we have
\begin{align*}
\frac{1}{2} \frac{d}{dt}q &=-\frac{1}{|u|^2} \nu  |A^2 u -q u|^2-\frac{1}{|u|^2} (B(u,u),A^2 u-q u) \notag\\
&=- \nu  |(A^2  -q) v|^2- |u| (B(v,v),(A^2 -q) v) .
\end{align*}
Consequently, by Cauchy-Schwarz inequality and Young's inequality, we infer
\begin{align*}
\frac{1}{2} \frac{d}{dt}q+\nu  |(A^2  -q) v|^2
&=- |u| (B(v,v),(A^2 -q) v) \notag\\
&\leq \frac{\nu}{2} |(A^2 -q) v|^2 + \frac{1}{2\nu}    |u|^2 |B(v,v)|^2.
\end{align*}
It follows that
\begin{align*}
\frac{d}{dt}q &\leq \frac{1}{\nu}    |u|^2 |B(v,v)|^2  =\frac{1}{\nu|u|^2}|B(u,u)|^2.
\end{align*}
Notice that $|B(u,u)|\leq |u|_{L^4} |\nabla u|_{L^4}\leq c |u|^{\frac{1}{2}}  \|u\|  |Au|^{\frac{1}{2}}$, by virtue of 2D Ladyzhenskaya's inequality $|u|_{L^4}\leq c|u|^{\frac{1}{2}}\|u\|^{\frac{1}{2}}.$
Then
\begin{align*}
\frac{d}{dt}q \leq \frac{c}{\nu|u|^2} |u|\|u\|^2 |Au|
\leq  c \nu^{-1} \lambda_1^{-\frac{1}{2}} |u|\|u\|  \left(\frac{|Au|^2}{|u|^2} \right)
=c \nu^{-1} \lambda_1^{-\frac{1}{2}}  |u|\|u\| q,
\end{align*}
where Poincar\'e inequality has been used.
This gives
\begin{align} \label{wohu-1}
q(t) \leq   q(t_0)  \exp\left\{c \nu^{-1} \lambda_1^{-\frac{1}{2}} \int_{t_0}^t |u| \|u\| ds \right\}  , \text{\;\;for\;\;}  t\geq t_0>0.
\end{align}

Now we estimate $\int_{t_0}^t |u| \|u\| ds$. Indeed, by (\ref{da}), one has
\begin{align}    \label{tst}
\int_{t_0}^t |Au|^2 ds =\frac{1}{2\nu} \left( |u(t_0)|^2-|u(t)|^2 \right).
\end{align}
Therefore, by Poincar\'e inequality, Cauchy-Schwarz inequality, as well as the estimates (\ref{dst}) and (\ref{tst}), we deduce
\begin{align*}
\int_{t_0}^t |u| \|u\| ds
&   \leq  \lambda^{-\frac{1}{2}}  \int_{t_0}^t |u| |Au| ds \notag\\
&\leq  \lambda^{-\frac{1}{2}} \left(\int_{t_0}^t |u|^2 ds\right)^{\frac{1}{2}}   \left(\int_{t_0}^t |Au|^2 ds\right)^{\frac{1}{2}} \notag\\
&\leq \lambda^{-\frac{1}{2}}|u(t_0)|
\left(   (2\nu  \lambda_1^2    ) ^{-1}  (1-e^{-2\nu\lambda_1^2(t-t_0)})        (|u(t_0)|^2-|u(t)|^2)        \right)^{\frac{1}{2}} \notag\\
&\leq (2\nu)^{-\frac{1}{2}} \lambda_1^{-\frac{3}{2}}  |u(t_0)|^2.
\end{align*}
Substituting this estimate into (\ref{wohu-1}) gives
\begin{align*}
q(t) \leq  q(t_0)     \exp{\left\{      c \nu^{-\frac{3}{2}} \lambda_1^{-2}   |u(t_0)|^2             \right \} }             .
\end{align*}
That is
\begin{align*}
|Au(t)|^2 \leq  \frac{|Au(t_0)|^2}{|u(t_0)|^2} |u(t)|^2   \exp{\left\{      c \nu^{-\frac{3}{2}} \lambda_1^{-2}   |u(t_0)|^2             \right \} }             ,
\end{align*}
and along with the energy identity, we obtain
\begin{align*}
\frac{d}{dt}|u(t)|^2 +2 \nu   \frac{|Au(t_0)|^2}{|u(t_0)|^2} |u(t)|^2
\exp{\left\{c \nu^{-\frac{3}{2}} \lambda_1^{-2}   |u(t_0)|^2 \right \} }
\geq 0.
\end{align*}
It follows that
\begin{align*}
|u(t)|\geq |u(t_0)| e^{- b (t-t_0)},  \text{\;\;for\;\;} t\geq t_0>0,
\end{align*}
where $b=\nu  \frac{|Au(t_0)|^2}{|u(t_0)|^2}    \exp{\left\{c \nu^{-\frac{3}{2}} \lambda_1^{-2}   |u(t_0)|^2 \right \} }         $     .
\end{proof}

\bigskip
\noindent {\bf Acknowledgment.}
This work was supported in part by a grant of the ONR and by the NSF grants DMS--1109640 and DMS--1109645.


\begin{thebibliography}{99}

 \bibitem{Ablowitz} M. J. Ablowitz, H. Segur, Solitons and the inverse scattering transform, SIAM Studies in Applied Mathematics 4, Society for Industrial and Applied Mathematics (SIAM), Philadelphia, Pa., 1981.


\bibitem{Aranson-02} I. S. Aranson, L. Kramer, The world of the complex Ginzburg-Landau equation, Rev. Modern Phys. 74 (2002), 99-143.





\bibitem{B-I-T} A. V. Babin, A. A. Ilyin, E. S. Titi, On the
    regularization mechanism for the periodic Korteweg-de Vries equation,
  Comm. Pure Appl. Math. 64 (2011), 591-648.


\bibitem{BBM} T. B. Benjamin, J. L. Bona, J. J. Mahony, Model equations for long waves in nonlinear dispersive systems, Philos. Trans. Roy. Soc. London Ser. A 272 (1972), 47-78.




\bibitem{Bardos-Tartar}  C. Bardos, L. Tartar, Sur l'unicit\'e r\'etrograde des \'equations paraboliques et quelques questions voisines, Arch. Rational Mech. Anal. 50 (1973), 10-25.



\bibitem{Ball-blowup} J. M. Ball, Remarks on blow-up and nonexistence theorems for nonlinear evolution equations, Quart. J. Math. Oxford Ser. (2) 28 (1977), 473-486.


\bibitem{book-Birnir}  B. Birnir, The Kolmogorov-Obukhov theory of turbulence, A mathematical theory of turbulence, Springer Briefs in Mathematics, Springer, New York, 2013.

\bibitem{Boffetta}  G. Boffetta, Energy and enstrophy fluxes in the double cascade of two-dimensional
turbulence, J. Fluid Mech. 589 (2007), 253-260.



\bibitem{Bou-1} J. Bourgain,
Fourier transform restriction phenomena for certain lattice subsets and applications to nonlinear evolution equations. I. Schrödinger equations,
Geom. Funct. Anal. 3 (1993), 107-156.


\bibitem{Bou-2} J. Bourgain, Fourier transform restriction phenomena for
    certain lattice subsets and applications to nonlinear evolution
    equations. II. The KdV-equation, Geom. Funct. Anal. 3 (1993), 209-262.



\bibitem{Bronski}  J. C. Bronski, T. N. Gambill, Uncertainty estimates and $L_2$ bounds for the Kuramoto-Sivashinsky equation, Nonlinearity 19 (2006), 2023-2039.

\bibitem{Burgers} J. M. Burgers, Correlation problems in a one-dimensional model of turbulence, Nederl. Akad. Wetensch. Proc. 53 (1950), 247-260.



\bibitem{Cazenave-13}  T. Cazenave, F. Dickstein, F. B. Weissler, Finite-time blowup for a complex Ginzburg-Landau equation, SIAM J. Math. Anal. 45 (2013), 244-266.



\bibitem{Collet-Eckmann}   P. Collet, J.-P. Eckmann, H. Epstein, J. Stubbe, A global attracting set for the Kuramoto-Sivashinsky equation, Comm. Math. Phys. 152 (1993), 203-214.



\bibitem{Constantin-Doering-95} P. Constantin, C. R. Doering, Variational bounds on energy dissipation in incompressible flows. II. Channel flow, Phys. Rev. E (3) 51 (1995), 3192-3198.


\bibitem{Cons-Doer-95}   P. Constantin, C. R. Doering, Variational bounds in dissipative systems. Phys. D 82 (1995), 221-228.


\bibitem{C-F-book}   P. Constantin, C. Foias, Navier-Stokes equations, Chicago Lectures in Mathematics, University of Chicago Press, Chicago, IL, 1988.


\bibitem{backward-NSE} P. Constantin, C. Foias, I. Kukavica, A. J. Majda, Dirichlet quotients and 2D periodic Navier-Stokes equations. J. Math. Pures Appl. (9) 76 (1997), 125-153.

\bibitem{Dascaliuc-03} R. Dascaliuc, On backward-time behaviour of Burgers' original model for turbulence, Nonlinearity 16 (2003), 1945-1965.

\bibitem{Dascaliuc-05} R. Dascaliuc, On backward-time behavior of the solutions to the 2-D space periodic Navier-Stokes equations, Ann. Inst. H. Poincar\'e Anal. Non Lin\'eaire 22 (2005), 385-401.


\bibitem{Doering-Constantin-92}   C. R. Doering, P. Constantin, Energy Dissipation in Shear Driven Turbulence, Phys. Rev. Lett. 69 (1992), 1648-1651.


\bibitem{Doering-Constantin-94}   C. R. Doering, P. Constantin, Variational bounds on energy dissipation in incompressible flows: shear flow, Phys. Rev. E (3) 49 (1994), 4087-4099.



 \bibitem{Doering-88} C. R. Doering, J. D. Gibbon, D. D. Holm, B. Nicolaenko, Low-dimensional behaviour in the complex Ginzburg-Landau equation, Nonlinearity 1 (1988), 279-309.



\bibitem{Drazin} P. G. Drazin, R. S. Johnson, Solitons: an introduction, Cambridge Texts in Applied Mathematics, Cambridge University Press, Cambridge, 1989.



\bibitem{Foias-Holm-Titi}  C. Foias, D. D. Holm, E. S. Titi,
The three dimensional viscous Camassa-Holm equations, and their relation to the Navier-Stokes equations and turbulence theory,  J. Dynam. Differential Equations 14 (2002), 1-35.


\bibitem{Lorenz} C. Foias, M. S. Jolly, On the behavior of the Lorenz equation backward in time, J. Differential Equations 208 (2005), 430-448.



\bibitem{book-Foias} C. Foias, O. Manley, R. Rosa, R. Temam, Navier-Stokes equations and turbulence, Encyclopedia of Mathematics and its Applications 83, Cambridge University Press, Cambridge, 2001.



\bibitem{Foias-Saut}  C. Foias, J.-C. Saut, Asymptotic behavior, as t$\rightarrow \infty$, of solutions of Navier-Stokes equations and nonlinear spectral manifolds, Indiana Univ. Math. J. 33 (1984), 459-477.

\bibitem{book-Frisch}  U. Frisch, Turbulence: The Legacy of A.N. Kolmogorov, Cambridge University Press, Cambridge, 1995.



\bibitem{Gardner-67} C. S. Gardner, J. M. Greene, M. D. Kruskal, R. M. Miura, Method for solving the Korteweg-deVries equation, Phys. Rev. Lett. 19 (1967), 1095-1097.


\bibitem{Gardner-74} C. S. Gardner, J. M. Greene, M. D. Kruskal, R. M. Miura, Korteweg-deVries equation and generalization. VI. Methods for exact solution,
Comm. Pure Appl. Math. 27 (1974), 97-133.



\bibitem{Otto-05} L. Giacomelli, F. Otto, New bounds for the Kuramoto-Sivashinsky equation, Comm. Pure Appl. Math. 58 (2005), 297-318.


\bibitem{Ghi}  J. M. Ghidaglia, Weakly damped forced Korteweg-de Vries equations behave as a finite-dimensional dynamical system in the long time, J. Differential Equations 74 (1988), 369-390.


\bibitem{Ghi-NLS} J. M. Ghidaglia, Finite-dimensional behavior for weakly damped driven Schr\"odinger equations. Ann. Inst. H. Poincar\'e Anal. Non Lin\'eaire 5 (1988), 365-405.


\bibitem{Gou-NLS} O. Goubet, Regularity of the attractor for a weakly damped nonlinear Schr\"odinger equation,  Appl. Anal. 60 (1996), 99-119.

\bibitem{Gou}  O. Goubet, Asymptotic smoothing effect for weakly damped forced Korteweg-de Vries equations, Discrete Contin. Dynam. Systems 6 (2000), 625-644.





\bibitem{Goodman-94} J. Goodman, Stability of the Kuramoto-Sivashinsky and related systems, Comm. Pure Appl. Math. 47 (1994), 293-306.


\bibitem{Guo-Simon-Titi} Y. Guo, K. Simon, E. S. Titi, Global well-posedness of a system of nonlinearly coupled KdV equations of Majda and Biello, Commun. Math. Sci., in press.



\bibitem{Johnson-73} R. S. Johnson, On the development of a solitary wave moving over an uneven bottom, Math. Proc. Cambridge Philos. Soc. 73 (1973), 183-203.



\bibitem{Jolly-Sadigov-Titi} M. S. Jolly, T. Sadigov, E. S. Titi, A determining form for the damped driven nonlinear Schr\"odinger equation - Fourier modes case, arXiv: 1406.2626.


\bibitem{Kakutani} T. Kakutani, Effect of an uneven bottom on gravity waves, J. Phys. Soc. Japan. 30 (1971), 272-276.





\bibitem{Kaup-Newell}  D. J. Kaup, A. C. Newell, Solitons as particles, oscillators, and in slowly changing media: a singular perturbation theory, Proc. R. Soc. Lond. A. 361 (1978), 413-446.




\bibitem{Kolmogorov-41}  A. N. Kolmogorov, The local structure of turbulence in incompressible viscous fluid for very large Reynolds numbers, Dokl. Akad. Nauk SSSR 30 (1941), 301-305.



\bibitem{Kolmogorov-41'}   A. N. Kolmogorov, Dissipation of energy in a locally isotropic turbulence, Dokl. Akad. Nauk SSSR 32 (1941), 16-18.



\bibitem{KdV} D. J. Korteweg, G. de Vries, On the change of form of long waves advancing in a rectangular canal, and on a new type of long stationary waves, Philosophical Magazine 39 (1895), 422-443.




\bibitem{Kraichnan} R. H. Kraichnan, Inertial ranges in two-dimensional turbulence, Phys. Fluids 10 (1967), 1417-1423.


\bibitem{Kraichnan'} R. H. Kraichnan, Inertial-range transfer in two- and three-dimensional turbulence, J. Fluid Mech. 47 (1971), 525-535.



\bibitem{KS-backward-92} I. Kukavica, On the behavior of the solutions of the Kuramoto-Sivashinsky equation for negative time, J. Math. Anal. Appl. 166 (1992), 601-606.



\bibitem{Kukavica-99}   I. Kukavica, Self-similar variables and the complex Ginzburg-Landau equation, Comm. Partial Differential Equations 24 (1999), 545-562.



\bibitem{KS-backward-05} I. Kukavica, M. Malcok, Backward behavior of solutions of the Kuramoto-Sivashinsky equation, J. Math. Anal. Appl. 307 (2005), 455-464.


\bibitem{Liu} S. Liu, Z. Fu, S. Liu, Q. Zhao, Jacobi elliptic function expansion method and periodic wave solutions of nonlinear wave equations, Phys. Lett. A 289 (2001), 69-74.



\bibitem{Larios-Titi-10}  A. Larios, E. S. Titi, On the higher-order global regularity of the inviscid Voigt-regularization of three-dimensional hydrodynamic models, Discrete Contin. Dyn. Syst. Ser. B 14 (2010), 603-627.




\bibitem{Larios-Titi-14}  A. Larios, E. S. Titi, Higher-order global regularity of an inviscid Voigt-regularization of the three-dimensional inviscid resistive magnetohydrodynamic equations, J. Math. Fluid Mech. 16 (2014), 59-76.



\bibitem{Majda-Bertozzi} A. J. Majda, A. L. Bertozzi, Vorticity and Incompressible Flow, Cambridge Texts in Applied Mathematics, Cambridge University Press, 2001.


\bibitem{Masmoudi-08}   N. Masmoudi, H. Zaag, Blow-up profile for the complex Ginzburg-Landau equation, J. Funct. Anal. 255 (2008), 1613-1666.


\bibitem{Moon-82}   H. T. Moon, P. Huerre, L. G. Redekopp,  Three-frequency motion and chaos in the Ginzburg-Landau equation, Phys. Rev. Lett. 49 (1982), 458-460.

\bibitem{Nazarenko}  S. Nazarenko, Wave Turbulence, Lecture Notes in Physics 825, Springer, Heidelberg, 2011.


\bibitem{KS-84} B. Nicolaenko, B. Scheurer, R. Temam, Some global dynamical properties of the Kuramoto-Sivashinsky equations: nonlinear stability and attractors, Phys. D 16 (1985), 155-183.


\bibitem{Obukhov}  A. M. Obukhov, On the distribution of energy in the spectrum of turbulent flow, Dokl. Akad. Nauk SSSR 32 (1941), 22-24.


\bibitem{Oliver-Titi} M. Oliver, E. S. Titi, Analyticity of the attractor and the number of determining nodes for a weakly damped driven nonlinear Schr\"odinger equation, Indiana Univ. Math. J. 47 (1998), 49-73.


\bibitem{Otto-09} F. Otto, Optimal bounds on the Kuramoto-Sivashinsky equation, J. Funct. Anal. 257 (2009), 2188-2245.




\bibitem{Rot-08}  V. Rottsch\" afer, Multi-bump, self-similar, blow-up solutions of the Ginzburg-Landau equation, Phys. D 237 (2008), 510-539.




\bibitem{Richardson}    L. F. Richardson, Atmospheric diffusion shown on a distance-neighbour graph, Proc. R. Soc. Lond. Ser. A Math. Phys. Eng. Sci. 110 (1926), 709-737.



\bibitem{Tbook} R. Temam, Infinite-dimensional dynamical systems in mechanics and physics, Second edition, Applied Mathematical Sciences 68, Springer-Verlag, New York, 1997.



\bibitem{Vukadinovic-02} J. Vukadinovi\'c, On the backwards behavior of the solutions of the 2D periodic viscous Camassa-Holm equations, J. Dynam. Differential Equations 14 (2002), 37-62.



\bibitem{Vukadinovic-04} J. Vukadinovi\'c, Density of global trajectories for filtered Navier-Stokes equations. Nonlinearity 17 (2004), 953-974.




\bibitem{Wang-NLS} X. Wang, An energy equation for the weakly damped driven nonlinear Schr\"odinger equations and its application to their attractors,
Phys. D 88 (1995), 167-175.


\bibitem{book-Zakharov} V. E. Zakharov, V. S. L'vov, G. Falkovich, Kolmogorov spectra of turbulence I: Wave turbulence, Springer Verlag Series in Nonlinear Dynamics, Springer-Verlag, New York, 1992.





\bibitem{Zakharov}  V. E. Zakharov, A. B. Shabat, Exact theory of two-dimensional self-focusing and one-dimensional self-modulation of waves in nonlinear media,  Soviet Physics JETP 34 (1972), 62-69.; translated from Z. Eksper. Teoret. Fiz. 61 (1971), 118-134.





\end{thebibliography}
\end{document}